\documentclass[11pt]{amsart}
\usepackage{amssymb}
\usepackage{amsmath}
\usepackage{amsthm}
\newtheorem{theorem}{Theorem}[section]
\newtheorem{lemma}[theorem]{Lemma}
\newtheorem{remark}[theorem]{Remark}
\newtheorem{definition}[theorem]{Definition}

\newtheorem{corollary}[theorem]{Corollary}
\newtheorem{proposition}[theorem]{Proposition}

\newtheorem{lem-def}[theorem]{Lemma-Definition}
\DeclareRobustCommand\longtwoheadrightarrow
{\relbar\joinrel\twoheadrightarrow}

\newcommand{\hooklongrightarrow}{\lhook\joinrel\longrightarrow}
\newcommand\bigZero{\mbox{\Large$0$}}

\newcommand{\Z}{\mathbb Z}
\newcommand{\Q}{\mathbb Q}

\newcommand{\V}{\mathbb{V}}

\def\op{\operatorname}

\def\al{\alpha}
\def\as#1{\renewcommand\arraystretch{#1}}

\def\cc{{\mathcal C}}

\def\di{\op{diag}}

\def\diso{\lower.4ex\hbox{$\downarrow$}\raise.4ex\hbox{\mbox{\scriptsize
$\wr$}}}

\def\e{\medskip}

\def\ep#1{\exp(\Pi i#1)}
\def\ep{\epsilon}

\def\g{\Gamma}
\def\ga{\gamma}

\def\gen#1{\big\langle\, {#1} \,\big\rangle}

\def\gg{\mathcal{G}}
\def\ggm{\mathcal{G}_\mu}
\def\ggminf{\mathcal{G}_{\minf}}
\def\ggmp{\mathcal{G}_{\mu'}}

\def\ggm{\mathcal{G}_\mu}

\def\gm{\g_\mu}
\def\gmf{\g_\mu^{\operatorname{fg}}}
\def\gp{\mathfrak{p}}

\def\gq{\g_\Q}
\def\gr{\operatorname{gr}}
\def\gv{\operatorname{gr}_v(K)}

\def\hm{H_\mu}
\def\hmp{H_{\mu'}}
\def\hmz{H_{\mu_0}}

\def\imp{\,\Longrightarrow\,}
\def\im{\op{Im}}

\def\iso{\ \lower.3ex\hbox{\as{.08}$\begin{array}{c}\lra\\\mbox{\tiny $\sim\,$}\end{array}$}\ }

\def\k{\op{Ker}}
\def\ka{\kappa}

\def\km{k_\mu}

\def\kpm{\op{KP}(\mu)}

\def\kx{K[x]}

\def\kxfg{K[x]_{\mu\op{-att}}}

\def\la{\lambda}

\def\lc{\op{lc}}

\def\lg{l\raise.6ex\hbox to.2em{\hss.\hss}l}

\def\lra{\,\longrightarrow\,}

\def\m{{\mathfrak m}}

\def\minf{\mu_{-\infty}}

\def\mmu{\mid_\mu}
\def\mn{\op{Min}}

\def\mx{\op{Max}}

\def\nph{N_{\mu,\phi}}
\def\nphm{N_{\mu',\phi}}
\def\npphm{N^{\mbox{\tiny pp}}_{\mu,\phi}}

\def\oo{\mathcal{O}}
\def\orb{\hbox to  .3em{$\backslash$}\backslash}
\def\ord{\op{ord}}

\def\ppa{\mathcal{P}_{\alpha}}
\def\pset{\mathcal{P}}

\def\qgq{\mathbb{Q}\times\gq}

\def\Rh{\widehat{R}}
\def\rr{\mathcal{R}}
\def\rrm{\mathcal{R}_\mu}

\def\sii{\,\Longleftrightarrow\,}
\def\slp{\op{sl}}
\def\smu{\sim_\mu}

\def\Min{\mathrm{Min}}

\def\Gi#1{\iota_{#1,1}\,\Z \oplus\cdots\oplus \iota_{{#1},k}\,\Z}
\def\lcm{\mathrm{lcm}}

\newcounter{cs}
\stepcounter{cs}
\newcommand{\casos}{\begin{itemize}}
\newcommand{\fcasos}{\end{itemize}\setcounter{cs}{1}}

\newfont{\tit}{cmr12 scaled \magstep3}

\title[Computation of residual polynomial operators]{Computation of residual polynomial operators of inductive valuations}

\makeatletter
\@namedef{subjclassname@2010}{%
  \textup{2010} Mathematics Subject Classification}

\author[Moraes de Oliveira]{Nath\'alia Moraes de Oliveira}
\address{Departament de Matem\`{a}tiques,
         Universitat Aut\`{o}noma de Barcelona,
         Edifici C, E-08193 Bellaterra, Barcelona, Catalonia, Spain}
\email{noliveira@mat.uab.cat,\quad nart@mat.uab.cat}

\author[Nart]{Enric Nart}
\thanks{Partially supported by grant MTM2016-75980-P from the Spanish MEC}
\date{}

\begin{document}

\begin{abstract}
Let $(K,v)$ be a valued field, and $\mu$ an inductive valuation on $\kx$ extending $v$. Let $\ggm$ be the graded algebra of $\mu$ over $\kx$, and $\kappa$ the maximal subfield of the subring of $\ggm$ formed by the homogeneous elements of degree zero.

In this paper, we find an algorithm to compute the field $\kappa$ and the residual polynomial operator $R_\mu\colon \kx\to\kappa[y]$, where $y$ is another indeterminate, without any need to perform computations in the graded algebra.

This leads to an OM algorithm to compute the factorization of separable defectless polynomials over henselian fields.
\end{abstract}

\maketitle

\section*{Introduction}

In a pioneering work along the 1920s, \O. Ore conjectured the existence of an algorithm to compute prime ideal decomposition in the number field defined by an irreducible polynomial $f\in\Q[x]$, based in the iteration of two major steps \cite{ore1,ore2}:  
\begin{itemize}
\item Computation of Newton polygons  of $\phi$-expansions of $f$, with respect to key polynomials $\phi$ for some auxiliary valuations on $\Q[x]$.
\item Factorization in the residue class fields of the above valuations, of certain residual polynomials associated to the sides of the Newton polygons.
\end{itemize}
In the 1930s, S. MacLane found an algorithm to solve a more general problem: given  a discrete rank one valuation $v$ on a field $K$, to find all its extensions to a   field extension $L=K[x]/f$ determined by an irreducible polynomial $f\in\kx$  \cite{mcla,mclb}. Any such extension may be identified to a valuation  $\mu$ on $\kx$ with support $f\kx$. MacLane's algorithm constructs a chain of augmentations of valuations on the polynomial ring $\kx$
\begin{equation}\label{chain0}
\mu_0\;<\;\mu_1\;<\;\cdots\;<\;\mu_i\;<\;\cdots\;<\;\mu
\end{equation}
which get arbitrarily close to $\mu$. In these augmentation steps, certain key polynomials for the valuations $\mu_i$  play an essential role. 
This algorithm can be reinterpreted as a polynomial factorization algorithm in $\hat{K}[x]$, where $\hat{K}$ is the completion of $K$ at $v$.

In the 40s, computers were born, and grew up in the 50s and 60s. Algorithmic arithmetic geometry appeared in the 60s and consolidated in the 70s as a mature mathematical discipline, but the ``old" ideas of Ore and MacLane were forgotten.

In the 80s, K. Okutsu stressed the role of key polynomials in the computation of integral bases of local fields \cite{Ok,okutsu}. In 1999, J. Montes developed  
certain residual polynomial operators which led to the design of a practical algorithm following the exact pattern that Ore had foreseen \cite{montes,bordeaux,HN}.
This algorithm is known as the \emph{OM algorithm}, named after Ore, MacLane, Okutsu and Montes.

The OM algorithm is very efficient in the resolution of many arithmetic-geometric tasks in number fields and function fields of algebraic curves. For an account on these applications, we refer the reader to the survey \cite{gen} and the references quoted there. 

MacLane's theory was generalized by M. Vaqui\'e to the case of an arbitrary valued field $(K,v)$, not necessarily rank-one nor discrete \cite{Vaq,Vaq2,Vaq3}. The graded algebra $\gg_{\mu}$  attached to a valuation $\mu$ on $\kx$ is crucial for the development of the theory. In this context, \emph{limit augmentations} and the corresponding \emph{limit key polynomials} appear as a new feature.


In analogy with the case of dimensions 0,1, MacLane-Vaqui\'e's theory should lead to the  developement of efficient algorithms for the resolution of arithmetic-geometric tasks involving valuations of function fields of algebraic varieties of higher dimension.  

A prototype of OM algorithm in this general frame was described by F.J. Herrera Govantes, W. Mahboub, M.A. Olalla Acosta and M. Spivakovsky in \cite{hmos}.
The main obstacle for this procedure to become a real algorithm is the existence of limit augmentations.
Another difficulty is the need to factorize polynomials in the graded algebra, which would be very onerous in practical implementations.

The first obstacle disappears if we deal with \emph{defectless} extensions of $K$. In this case, the auxiliary valuations $\mu_i$ in (\ref{chain0}) are all \emph{inductive} \cite{Vaq2,MN}; that is, no limit augmentations are involved in their construction.


In this paper, we present an OM algorithm to compute extensions of valuations of arbitrary rank to defectless field extensions. More precisely, for each extension of $v$ to an inductive valuation $\mu_r$ on $\kx$ of length $r$ in the chain (\ref{chain0}) of augmentations, we construct in a recursive  way a tower of field extensions  of the residue class field $k$ of $v$:
$$
k=\ka_0\;\subset\; \ka_1\;\subset\;\cdots\;\subset\; \ka_r,
$$
and residual polynomial operators 
$$
R_{\mu_i}\colon K[x]\lra \ka_i[y],\qquad 0\le i\le r,
$$
which facilitate the design of an OM algorithm completely analogous to the classical one.

In the construction of these objects we avoid any computation in the graded algebra.
In particular, these ideas provide an efficient polynomial factorization algorithm for separable defectless polynomials over henselian fields.


The definition of $R_{\mu_r}$ involves the choice of rational functions in $K(x)$ whose images in the graded algebra $\gg_{\mu_r}$ have certain prescribed properties. For our algorithmic approach we need to construct these rational functions for the operators $R_{\mu_0},\dots, R_{\mu_{r-1}}$ in a coherent way.  
To this end, we generalize the methods of \cite{ResidualIdeals}, where this problem was solved for discrete rank-one valuations.  

Section \ref{secIndVals} of the paper is devoted to the construction of the adequate rational functions in $K(x)$. In section \ref{secNewton}, we discuss Newton polygons. In section \ref{secRi}, we define the operators $R_{\mu_i}$ and give some applications.
Our main result is Theorem \ref{recursivecj}, where we prove the recursivity of the residual polynomial operators.




\section{Valuations on polynomial rings}\label{secVals}
\pagestyle{headings}

Let  $(K,v)$ be a valued field with non-trivial group $\g=v(K^*)$ and residue class field $k$.


Let $\mu$ be a valuation on the polynomial ring $\kx$, extending $v$. Let $\g_\mu$ be the value group, and denote by $k_{\mu}$ its residue class field.

We suppose that $\mu$ has trivial support; that is, $\mu$ is the restriction to $\kx$ of some valuation on the field $K(x)$.

For any $\alpha\in\g_\mu$, consider the abelian groups:
$$
\ppa=\{g\in A\mid \mu(g)\ge \alpha\}\supset
\ppa^+=\{g\in A\mid \mu(g)> \alpha\}.
$$    

The \emph{graded algebra of $\mu$ over $\kx$} is the integral domain:
$$
\ggm:=\gr_{\mu}(\kx)=\bigoplus\nolimits_{\alpha\in\g_\mu}\ppa/\ppa^+.
$$

Let $\Delta=\Delta_\mu=\pset_0/\pset_0^+\subset\ggm$ be the subring of homogeneous elements of degree zero. 
There are canonical injective ring homomorphisms: 
$$ k\hooklongrightarrow\Delta\hooklongrightarrow k_{\mu}.$$
In particular, $\Delta$ and $\ggm$ are equipped with a canonical structure of $ k$-algebra. \e

The initial term mapping $\hm\colon \kx\to \ggm$ is given by $\hm(0)=0$ and
$$\hm(g)= g+\pset_{\mu(g)}^+\in\pset_{\mu(g)}/\pset_{\mu(g)}^+, 
$$
if $g\ne0$. For all $g,h\in \kx$ we have $\hm(gh)=\hm(g)\hm(h)$ and
\begin{equation}\label{Hmu}
 \hm(g+h)=\hm(g)+\hm(h), \ \mbox{ if }\mu(g)=\mu(h)=\mu(g+h).
\end{equation}

The next definitions translate properties of the action of  $\mu$ on $\kx$ into algebraic relationships in the graded algebra $\ggm$.

\begin{definition}\label{mu}Let $g,\,h\in \kx$.

We say that $g,h$ are \emph{$\mu$-equivalent}, and we write $g\smu h$, if $\hm(g)=\hm(h)$. 

We say that $g$ is \emph{$\mu$-divisible} by $h$, and we write $h\mmu g$, if $\hm(h)\mid \hm(g)$ in $\ggm$. 

We say that $g$ is $\mu$-irreducible if $\hm(g)\ggm$ is a non-zero prime ideal. 

We say that $g$ is $\mu$-minimal if $g\nmid_\mu f$ for any non-zero $f\in \kx$ with $\deg f<\deg g$.
\end{definition}

The property of $\mu$-minimality admits a relevant characterization.

\begin{lemma}\label{minimal0}
Let $\phi\in \kx$ be a non-constant polynomial. Let 
$$
f=\sum\nolimits_{0\le s}a_s\phi^s, \qquad  a_s\in \kx,\quad \deg(a_s)<\deg(\phi) 
$$
be the canonical $\phi$-expansion of $f\in \kx$.
Then, $\phi$ is $\mu$-minimal if and only if
$$
\mu(f)=\mn\{\mu(a_s\phi^s)\mid 0\le s\},\quad \forall\,f\in \kx.
$$
\end{lemma}


\begin{definition}\label{RRmu}
Let $I(\Delta)$ be the set of ideals in $\Delta$, and consider the \emph{residual ideal operator}:
$$
\rr=\rrm\colon \kx\lra I(\Delta),\qquad g\mapsto \left(\hm(g)\ggm\right)\cap \Delta.
$$
\end{definition}

This operator $\rr$ translates questions about the action of  $\mu$ on $\kx$ into ideal-theoretic consi\-derations in the $k$-algebra $\Delta$.


\begin{definition} A \emph{MacLane-Vaqui\'e key polynomial} for $\mu$ is a monic polynomial in $\kx$ which is simultaneously $\mu$-minimal and $\mu$-irreducible. 

The set of key polynomials for $\mu$ will be denoted by $\kpm$.
\end{definition}

A key polynomial is necessarily irreducible in $\kx$.




For any positive integer $m$ we denote:
$$
\g_{\mu,m}=\left\{\mu(a)\in\gm\mid a\in K[x],\ a\ne0,\ \deg(a)<m\right\}\subset\gm.
$$

\begin{proposition}\label{vphi}
Take $\phi\in\kpm$ and let $m=\deg(\phi)$. Consider the maximal ideal $\gp=\phi\kx$, the field $K_\phi=\kx/\gp$,
and the onto mapping:
$$
v_\phi\colon K_\phi^*\longtwoheadrightarrow \g_{\mu,m},\qquad v_\phi(f+\gp)=\mu(a_0),\quad \forall f\in\kx\setminus\gp,
$$
where $a_0\in\kx$ is the common $0$-th coefficient of the $\phi$-expansion of all polynomials in the class $f+\gp$. Then,

\begin{enumerate}
\item The set $\g_{\mu,m}$ is a subgroup of $\gm$ and $\gen{\g_{\mu,m},\mu(\phi)}=\gm$.
\item The mapping $v_\phi$ is a valuation on $K_\phi$ extending $v$, with value group $\g_{\mu,m}$.
\end{enumerate}
\end{proposition}

Denote the maximal ideal, the valuation ring and the residue class field of $v_\phi$ by: 
$$\m_\phi\subset\oo_\phi\subset K_\phi,\qquad k_\phi=\oo_\phi/\m_\phi.$$

\begin{proposition}\label{maxideal}
If $\phi\in\kpm$, then 
$\rr(\phi)$ is the kernel of the onto homomorphism $$\Delta\longtwoheadrightarrow k_\phi,\qquad g+\pset^+_0\ \longmapsto\ \left(g+\gp\right)+\m_\phi.$$

In particular, $\rr(\phi)$ is a maximal ideal in $\Delta$.
\end{proposition}

\begin{proposition}\label{maxsubfield}
Let $\kappa\subset\Delta$ be the algebraic closure of $k$ in $\Delta$. Then, $\kappa$ is a subfield such that $\kappa^*=\Delta^*$. 
Moreover, if $\phi\in\kpm$ has minimal degree, the mapping $\kappa\hookrightarrow\Delta\twoheadrightarrow k_\phi$ is an isomorphism. 
\end{proposition}


The extension $\mu/v$ is \emph{commensurable} if $\g_\mu/\g$ is a torsion group. In this case, there is a canonical embedding
$$
\g_\mu\hooklongrightarrow \gq:=\g\otimes\Q.
$$
By composing $\mu$ with this embedding, we obtain a $\gq$-valued valuation 
$$\kx \longrightarrow \gq \cup \lbrace\infty\rbrace.$$
Conversely, any $\gq$-valued extension of $v$ to $\kx$ is commensurable over $v$.

Two commensurable extensions of $v$ are equivalent if and only if their corresponding $\gq$-valued valuations coincide.
Hence, we may identify the set of equivalence classes of comensurable valuations extending $v$ with the set
$$\V:=\V(K,v)=\lbrace \mu\colon \kx\lra \gq \cup \lbrace\infty\rbrace \mid \mu\mbox{ valuation, } \mu_{\vert K}=v\rbrace.$$

There is a natural partial ordering in the set $\V$:
$$\mu\le \mu' \quad\mbox{ if }\quad\mu(f) \le \mu'(f), \quad \forall\,f\in \kx. $$

\begin{theorem}\label{kstructure}\mbox{\null}
Suppose $\mu/v$ commensurable and $\kpm\ne\emptyset$. Take $\phi\in\kpm$ of minimal degree, and let  $e=\left(\gm\colon \g_{\mu,\deg(\phi)}\right)$. 
Let $\epsilon\in\ggm^*$ be a unit of degree $-e\mu(\phi)$. 

Then, $\xi=\hm(\phi)^e\epsilon\in\Delta$ is transcendental over $\kappa$, and
$$\Delta=\kappa[\xi],\qquad \op{Frac}(\Delta)=\kappa(\xi)\simeq \km,$$
the last isomorphism being induced by the canonical embedding $\Delta\hookrightarrow \km$.
\end{theorem}


\begin{definition}\label{muprima}
Let $\iota\colon \gm\hookrightarrow \g'$ be an order-preserving embedding of $\gm$ into another abelian ordered group.
Take $\phi\in \kpm$ and $\ga\in \g'$ any element such that $\mu(\phi)<\ga$. 

The \emph{augmented valuation} of $\mu$ with respect to these data is the mapping $$\mu'\colon\kx\rightarrow \g' \cup \left\{\infty\right\}$$ which assigms to any $f\in\kx$, with $\phi$-expansion $f=\sum_{0\le s}a_s\phi^s$, the value:
$$
\mu'(f)=\mn\left\{\mu'(a_s\phi^s)\mid 0\le s\right\}=\mn\left\{\mu(a_s)+s\ga\mid 0\le s\right\}.
$$
We use the notation $\mu'=[\mu;\phi,\ga]$.
\end{definition}

There is a canonical homomorphism of graded algebras:
$$\ggm\lra\gg_{\mu'},\qquad \hm(f)\longmapsto
\begin{cases}\hmp(f),& \mbox{ if }\mu(f)=\mu'(f),\\ 0,& \mbox{  if }\mu(f)<\mu'(f).\end{cases}$$

\begin{proposition}\label{extension} Let $\mu'=[\mu;\phi,\ga]$, and let $f\in \kx$ be a non-zero polynomial.
\begin{enumerate}
\item[(a)] The valuation $\mu'$ extends $v$ and satisfies $\mu(f)\le \mu'(f)$ for all $f\in\kx$. 

Equality holds if and only if $\phi\nmid_{\mu}f$. In this case,  $\hmp(f)$ is a unit in $\ggmp$.
\item[(b)] The kernel of the homomorphism $\ggm\to\gg_{\mu'}$ is the prime ideal $\hm(\phi)\,\ggm$.
\item[(c)] The polynomial $\phi$ is a key polynomial for $\mu'$ of minimal degree.
\item[(d)] The value group of $\mu'$ is \ $\g_{\mu'}=\gen{\g_{\mu,\deg(\phi)},\ga}\subset\g'$. \end{enumerate}
\end{proposition}



Propositions \ref{extension}, \ref{maxideal} and \ref{maxsubfield} provide a precise description of the kernel and image of the canonical homomorphism $\Delta_\mu\to\Delta_{\mu'}$. 

\begin{lemma}\label{imagedelta}
Let $\mu'=[\mu;\phi,\ga]$. Then,  
$$
\k(\Delta_\mu\to \Delta_{\mu'})=\rr(\phi),\qquad \im(\Delta_\mu\to \Delta_{\mu'})=\kappa_{\mu'}\hookrightarrow k_{\mu'},
$$
where $\kappa_{\mu'}\simeq k_\phi$ is the algebraic closure of $k$ simultaneously in $\Delta_{\mu'}$ and in $k_{\mu'}$. 

\end{lemma}




\section{Valuations of depth zero}\label{secDepth0}

Let us fix an order-preserving embedding of abelian ordered groups:
\begin{equation}\label{gembb}
\gq\hooklongrightarrow \left(\Z\times \gq\right)_{\op{lex}},\qquad \ga\longmapsto (0,\ga),
\end{equation}
where in $\left(\Z\times \gq\right)_{\op{lex}}$ we consider the lexicographical order.

Consider the following valuation on $\kx$ extending $v$:
$$
\minf\colon \kx\lra \left(\Z\times \g\right)_{\op{lex}}\cup\{\infty\},\qquad 
f\longmapsto \left(-\deg(f),v\left(\lc(f)\right)\right)
$$
where $\lc(f)\in K^*$ is the leading coefficient of a non-zero polynomial $f$.

Since $\g_{\minf}=\left(\Z\times \g\right)_{\op{lex}}$, the extension $\minf/v$ is incommensurable.


The following (trivial) observation is left to the reader.

\begin{lemma}\label{hminf}
Let $f,g\in\kx$ be non-zero polynomials.
\begin{enumerate}
\item[(1)] $f\sim_{\minf}\lc(f)x^{\deg(f)}$.
\item[(2)] $f\sim_{\minf}g\ \iff\ \deg(f)=\deg(g),\quad \lc(f)\sim_v\lc(g)$.
\end{enumerate}
\end{lemma}

The following result is an immediate consequence of Lemma \ref{hminf}.

\begin{corollary}\label{kpminf}
Let $y$ be an indeterminate, to which we assign degree $(-1,0)$. 

There is an  isomorphism of graded algebras 
$$\gv[y] \ \iso \ \ggminf,\qquad y  \longmapsto  H_{\minf}(x) .
$$
It induces an isomorphism $k\simeq\Delta_{\minf}\simeq k_{\minf}$ of $k$-algebras.

Moreover, $\op{KP}(\minf)=\left\{x+a\mid a\in K\right\}$ and all key polynomials are $\minf$-equivalent.
\end{corollary}



\begin{definition}
A valuation $\mu$ on $\kx$ is said to have \emph{depth zero} if it  is commensurable over $v$, and it is equivalent to some augmentation of $\minf$.  
\end{definition}

By Corollary \ref{kpminf}, a valuation of depth zero must be of the form:
$$
\mu_{x+a,\ga}=[\minf;\,x+a,\,(0,\ga)],\qquad a\in K,\quad\ga\in\gq.
$$

It is easy to check under what conditions two of these augmentations coincide: 
$$
\mu_{x+a,\ga}=\mu_{x+b,\ga_*}\ \iff\  v(a-b)\ge\ga=\ga_*.
$$

By Propositon \ref{extension}, the value group of these valuations is:
$$
\g_{\mu_{x+a,\ga}}=\gen{\{0\}\times \g,\,(0,\ga)}.
$$

By dropping the first (null) coordinate, we obtain equivalent valuations with values in $\gq$, thus  belonging to the space $\V$. We denote them with the same symbol $\mu_{x+a,\ga}$.

Their value group is $\g_{\mu_{x+a,\ga}}=\gen{\g,\,\ga}$ and they act as follows:
$$
\mu_{x+a,\ga}\colon \ \sum\nolimits_{0\le s}a_s(x+a)^s\ \longmapsto\ \mn\left\{v(a_s)+s\ga\mid0\le s\right\}.
$$

For instance, the depth-zero valuation $\mu_{x,0}$ is known as \emph{Gauss' valuation}.

\subsection*{Absolute minimality of $\minf$}
Thanks to the  embedding (\ref{gembb}),  we may compare the values of $\minf$ with those of the valuations in $\V$. We clearly have
\begin{equation}\label{minfmin}
\minf(f)\le \mu(f),\qquad \forall\,f\in\kx,\quad\forall\,\mu\in\V.
\end{equation}

In a certain sense, this property characterizes $\minf$.

\begin{theorem}\label{absmin}
Let  $\eta\colon \kx\to \g'\cup \left\{\infty\right\}$ be a valuation extending $v$, with respect to some order-preserving embedding $\iota\colon \g\hookrightarrow\g'$.

Then, after embedding $\;\gq\hookrightarrow\gq'$ by \ $\iota\otimes\Q$, the following conditions are equivalent.

\begin{enumerate}
\item[(1)] $\eta(f)\le \mu(f)$, \ for all $\,f\in\kx$ and all $\,\mu\in\V$.
\item[(2)] $\eta(x)< \ga$, \ for all  $\,\ga\in\gq$. 
\item[(3)] The valuation $\eta$ is equivalent to $\minf$.
\end{enumerate}
\end{theorem}

\begin{proof}
Suppose condition (1) is satisfied.
Then, our valuation $\eta$ satisfies:
$$
\eta(x)\le \mu_{x,\ga}(x)=\ga,\qquad \forall\,\ga\in\gq.
$$
Since $v$ is non-trivial, the group $\gq$ has no minimal element, and the last inequality must be strict. This proves (2).\e

Now, suppose that condition (2) is satisfied. We deduce immediately
\begin{equation}\label{xm}
\eta(x^m)< \ga< \eta(x^{-m}),\qquad \forall\,\ga\in\gq,\ \forall\,m\in\Z_{>0}. 
\end{equation}

This property implies:
\begin{equation}\label{simeta}
f\sim_{\eta} \lc(f)\,x^{\deg(f)},\quad \forall\,f\in\kx,\ f\ne0. 
\end{equation}
In fact, any two non-zero monomials $ax^m$, $bx^n$ of different degree, have different $\eta$-value, and the smallest value is that of the monomial of maximal degree:
\begin{equation}\label{ordering}
n<m\imp \eta(x^{m-n})<v(b/a)\imp \eta(ax^m)<\eta(bx^n).
\end{equation}

Hence, we get an order-preserving group isomorphism
$$
j\colon\left(\Z\times \g\right)_{\op{lex}}\lra \g_\eta,\qquad (m,\al)\longmapsto \al-\eta(x^m).
$$
In fact, (\ref{simeta}) shows that $j$ is onto. And, if $(m,\al)\in\k(j)$, then   (\ref{xm}) implies
$$
\eta(x^m)=\al\in\gq\  \imp\ m=0\ \imp\ \al=0.
$$
Thus $j$ is a group isomorphism.
Also, $j$ preserves the ordering:
$$(n,v(a))\le (m,v(b)) \ \imp\ j(n,v(a))=\eta(ax^{-n})\le \eta(bx^{-m})=j(m,v(b)).
$$
If $n<m$, this inequality follows from (\ref{ordering}). If $n=m$, it follows from $\eta_{\mid K}=v$. 

This proves (3), because the following diagram commutes:
\begin{equation}\label{jiso}
\as{1.3}
\begin{array}{ccc}
\left(\Z\times \g\right)_{\op{lex}}&\stackrel{j}\iso&\quad\g_\eta\\
\quad\lower.4ex\hbox{\scriptsize$\minf$}\!\!\nwarrow&&\!\!\!\!\nearrow\lower.4ex\hbox{\scriptsize$\eta$}\\
&K(x)^*&
\end{array}
\end{equation}
In fact, for any non-zero $f\in\kx$, equation  (\ref{simeta}) shows that
$$
\eta(f)=\eta\left(\lc(f)\,x^{\deg(f)}\right)=j\left(-\deg(f),v\left(\lc(f)\right)\right)=j(\minf(f)).
$$

Finally, suppose that $\eta$ and $\minf$ are equivalent. That is, there exists an order-preserving isomorphism $j$ such that diagram (\ref{jiso}) commutes.

Since $\eta_{\mid K}=v={\minf}_{\mid K}$, the isomorphism $j\otimes\Q$ maps $\{0\}\times\gq$ into $\iota(\gq)$. Hence, by applying $j$ to the inequalities in (\ref{minfmin}), we obtain the inequalities in item (1).
\end{proof}




\section{Inductive valuations}\label{secIndVals}

A valuation $\mu$ on $\kx$ extending $v$ is said to be \emph{inductive} if it is attained after a finite number of augmentation steps starting with the minimal valuation $\minf$:
\begin{equation}\label{depth}
\minf\stackrel{\phi_0,\ga_0}\lra\  \mu_0\ \stackrel{\phi_1,\ga_1}\lra\  \mu_1\ \stackrel{\phi_2,\ga_2}\lra\ \cdots
\ \stackrel{\phi_{r-1},\ga_{r-1}}\lra\ \mu_{r-1} 
\ \stackrel{\phi_{r},\ga_{r}}\lra\ \mu_{r}=\mu,
\end{equation}
with values $\ga_0,\dots,\ga_r\in\gq$, and intermediate valuations  
$$\mu_0=\mu_{\phi_0,\ga_0}, \qquad\mu_i=[\mu_{i-1};\phi_i,\ga_i], \quad 1\leq i\leq r.
$$

Inductive valuations are commensurable over $v$, because we do not consider $\minf$ to be an inductive valuation. 
The integer $r\ge0$ is the \emph{length} of the chain (\ref{depth}).  \e

By Proposition \ref{extension}, every $\phi_{i}$ is a key polynomial for $\mu_i$ of minimal degree. 

Since every $\phi_{i+1}$ is $\mu_i$-minimal, \cite[Prop. 3.7]{KeyPol} shows that
$$
1=\deg(\phi_0)\mid \deg(\phi_1)\mid\cdots\mid\deg(\phi_{r-1})\mid\deg(\phi_r).
$$

\begin{lemma}\label{stable}
For a chain of augmentations as in {\rm (\ref{depth})}, take $f\in \kx$ such that $\phi_i\nmid_{\mu_{i-1}}f$ for some $1\le i\le r$. Then, 
$\mu_{i-1}(f)=\mu_{i}(f)=\cdots=\mu_r(f)$. 
\end{lemma}

\begin{proof}
By Proposition \ref{extension}, $\mu_{i-1}(f)=\mu_{i}(f)$ and $H_{\mu_i}(f)$ is a unit. Thus, $\phi_{i+1}\nmid_{\mu_{i}}f$ since $H_{\mu_i}(\phi_{i+1})$ is a prime element in $\gg_{\mu_i}$; hence, the argument may be iterated.
\end{proof}

\subsection{MacLane chains of valuations}\label{subsecML}
Let us impose a technical condition on the augmentations to ensure that the value groups $\g_{\mu_0},\dots,\g_{\mu_r}$ form a chain.

For a chain as in (\ref{depth}) and any index $0\le i<r$, \cite[Prop. 6.6]{KeyPol} shows that:
$$\phi_{i+1}\mid_{\mu_{i}}\phi_{i}  \quad\sii\quad
\phi_{i+1}\sim_{\mu_i}\phi_i \quad\imp\quad\deg(\phi_i)=\deg(\phi_{i+1}).
$$

\begin{definition}
A chain of augmented valuations as in (\ref{depth}) is called a \emph{MacLane chain} if it satisfies $\phi_{i+1}\nmid_{\mu_{i}}\phi_{i}$ for all $0\le i<r$.

If $\deg(\phi_0)<\cdots<\deg(\phi_r)$, we say that it is an \emph{optimal} MacLane chain.
\end{definition}

Obviously, the truncation of a MacLane chain at the $i$-th node 
is a MacLane chain of the intermediate valuation $\mu_i$. 

As an immediate application of Lemma \ref{stable}, in a MacLane chain one has:
$$
\mu(\phi_i)=\mu_i(\phi_i)=\ga_i,\qquad 0\le i\le r.
$$

All inductive valuations admit optimal MacLane chains, as the next result shows.

\begin{lemma}\cite[Sec. 1.2]{Vaq}\label{existence}
Consider a chain of two augmented valuations
$$
\mu\ \stackrel{\phi,\ga}\lra\  \mu'\ \stackrel{\phi_*,\ga_*}\lra\ \mu_*
$$
with \,$\deg(\phi_*)=\deg(\phi)$. Then, $\phi_*$ is a key polynomial for $\mu$, and $\mu_*=[\mu;\phi_*,\ga_*]$.
\end{lemma}

Thus, the length of a MacLane chain of $\mu$ is not an intrinsic datum of $\mu$. 
However, there is a strong unicity statement if we consider only optimal MacLane chains. 

The next result follows from \cite[Prop. 3.6]{ResidualIdeals}.
Although that paper deals with rank one valuations, the arguments used in the proof are valid for arbitrary valuations.

\begin{proposition}\label{unicity}
Consider an optimal MacLane chain as in {\rm (\ref{depth})} and another optimal MacLane chain
$$
\minf\ \stackrel{\phi^*_0,\ga^*_0}\lra\  \mu^*_0\ \stackrel{\phi^*_1,\ga^*_1}\lra\ \cdots
\ \lra\ \mu^*_{t-1} 
\ \stackrel{\phi^*_{t},\ga^*_{t}}\lra\ \mu^*_{t}=\mu^*.
$$
Then, $\mu=\mu^*$ if and only if $r=t$ and
$$
\deg(\phi_i)=\deg(\phi^*_i), \quad \mu_{i-1}(\phi_i-\phi^*_i)\ge\ga_i=\ga^*_i \ \mbox{ for all }\ 0\le i \le r.
$$
In this case, we also have $\mu_i=\mu^*_i$ for all $0\le i \le r$. 
\end{proposition}

Therefore, in any optimal MacLane chain of an inductive valuation $\mu$ as in (\ref{depth}), the intermediate valuations $\mu_0,\dots,\mu_{r-1}$, the values $\ga_0,\dots,\ga_r\in\gq$ and the positive integers $\deg(\phi_0),\dots,\deg(\phi_r)$, are intrinsic data of $\mu$.

The \emph{depth} of $\mu$ is the length $r$ of any optimal MacLane chain of $\mu$.  


The next result is an immediate consequence of \cite[Cor. 6.4]{KeyPol}.

\begin{lemma}\label{groupchain}
Consider a MacLane chain as in {\rm (\ref{depth})}, and denote $\mu_{-1}:=v$.

The value groups of the valuations $\mu_i$ and $v_{\phi_i}$ are:
$$\g_{v_{\phi_i}}=\g_{\mu_{i-1}}\subset \g_{\mu_{i}}=\gen{\g_{\mu_{i-1}},\ga_i}, \quad 0\le i\le r.$$ 

\end{lemma}

In particular, we have a chain of ordered groups 
$$\g_{\mu_{-1}}:=\g\subset \g_{\mu_0}\subset\cdots\subset \g_{\mu_r}=\g_\mu,$$ and every quotient $\g_{\mu_i}/\g_{\mu_{i-1}}$  
  is a finite cyclic group generated by $\ga_i$. \e

From now on, we shall use the following notation, for all $0\le i\le r$. 
$$
m_i=\deg(\phi_i),\qquad e_i=\left(\g_{\mu_i}\colon \g_{\mu_{i-1}}\right),\qquad 
h_i=e_i\ga_i\in\g_{\mu_{i-1}}.
$$

The identification $\g_{\mu_{i-1}}=v\left(K^*_{\phi_i}\right)$ yields a computation of the ramification index of the extension $K_{\phi_i}/K$ in terms of these data: 
\begin{equation}\label{ephi}
e(\phi_i)=\left(\g_{\mu_{i-1}}\colon \g\right)=e_0\cdots e_{i-1}.
\end{equation}

\subsection*{Chain of homomorphisms between the graded algebras}
Some more data supported by a MacLane chain are derived from the chain of homomorphisms:
$$
\gg_{\mu_0}\lra\gg_{\mu_1}\lra\quad\cdots\quad \lra\gg_{\mu_{r-1}}\lra\gg_{\mu}.
$$

Denote $\Delta_i=\Delta_{\mu_i}$ for $0\le i\le r$. By Lemma \ref{imagedelta}, there is a sequence of fields
$$
 \ka_0=\op{Im}( k\to\Delta_0), \qquad   \ka_i=\op{Im}(\Delta_{i-1}\to\Delta_i),\quad 1\le i\le r,
$$
where $ \ka_i$ is isomorphic to the residue class field $ k_{\phi_i}$ of the extension $K_{\phi_i}/K$ determined by $\phi_i$. In particular, $ \ka_i$ is a finite extension of $ k$. 


We identify $ k$ with $ \ka_0$, and each field $ \ka_i\subset\Delta_i$ with its image under the canonical map $ \Delta_i\to\Delta_j$ for $j\ge i$. Thus, we consider as inclusions the canonical embeddings: 
$$
k= \ka_0\subset  \ka_1\subset\cdots\subset  \ka_r. 
$$

The residual degree of the extension $K_{\phi_i}/K$ can be computed as:
\begin{equation}\label{fphi}
f(\phi_i)=\left[ \ka_i\colon  \ka_0\right]=f_0\cdots f_{i-1},\qquad f_{i-1}=[ \ka_{i}\colon  \ka_{i-1}], \quad 1\le i \le r.
\end{equation}

\subsection{Chain of finitely-generated value subgroups}\label{subsecChainFG}
If the group $\gm$ is not finitely-generated, the homomorphism 
$\mu\colon K(x)^* \twoheadrightarrow \gm$ does not admit a section.

In this section, the index $i$ takes any integer value $0\le i\le r$. 

Let us choose finitely-generated subgroups $\g_i\subset\g_{\mu_i}$. In section \ref{subsecRat}, we shall construct a section $\g_i\hookrightarrow K(x)^*$ of $\mu_i$. 

Since $\g_{\mu_i}=\g+\gen{\ga_0,\dots,\ga_i}$, there exist $\al_i\in\g$ such that $h_i\in\al_i+\gen{\ga_0,\dots,\ga_{i-1}}$. 

\begin{definition}\label{defgf}
We fix a finitely-generated subgroup $\g_{-1}\subset\g$ containing $\al_0,\dots,\al_r$. This choice determines finitely-generated subgroups
$$
\g_i:=\g_{\mu_i}^{\op{fg}}:=\gen{\g_{i-1},\ga_i}=\gen{\g_{-1},\ga_0,\dots,\ga_i}\subset \g_{\mu_i}.
$$
By construction, $h_i=e_i\ga_i\in\g_{i-1}$ for all $i$. 
\end{definition}

These finitely-generated groups form a chain $\g_{-1}\subset \g_0\subset\cdots\subset \g_r$,
with
\begin{equation}\label{intersection}
\g_i\cap\g_{\mu_{i-1}}=\g_{i-1},\qquad \g_i/\g_{i-1}\simeq \Z/e_i\Z. 
\end{equation}

Finitely-generated ordered groups are free as $\Z$-modules. Take a basis of $\g_{-1}$:
$$
\g_{-1}=\Gi{0}.
$$
From this basis, we shall derive specific bases for the other groups:
$$
\g_i=\Gi{i+1},
$$
by a recurrent procedure. To this end, let us write:
\begin{equation}\label{u}
\ga_i=\left(\iota_{i,1}\cdots\iota_{i,k}\right)\mathbf{u}\in(\g_{i-1})_\Q,\qquad \mathbf{u}=\left(h_{i,1}/e_{i,1} \dots\, h_{i,k}/e_{i,k} \right)^t\in\Q^{k\times 1}
\end{equation}
with $h_{i,j},\,e_{i,j}$ coprime integers with $e_{i,j}>0$, for all $1\le j\le k$. \e

\noindent{\bf Notation. }Denote $e'_{i,1}=d_{i,1}=1$, and
$$
e'_{i,j}=e_{i,1} \cdots e_{i,j-1}/\left(d_{i,1} \cdots d_{i,j-1}\right),  \qquad    d_{i,j}=\gcd\left(e_{i,j},e'_{i,j}\right), \quad 1< j \leq k.
$$
There are uniquely determined integers $\ell_{i,j}$, $\ell'_{i,j}$ satisfying B\'ezout identities 
\begin{equation}\label{bezout}
\ell_{i,j}h_{i,j}e'_{i,j}+\ell'_{i,j}e_{i,j}= d_{i,j} , \quad \ 0\leq \ell_{i,j} < e_{i,j}/d_{i,j} ,  \quad 1\leq j \leq k,
\end{equation}

\begin{lemma}\label{lcm}
$e_{i,1} \cdots e_{i,k}/\left(d_{i,1}\cdots d_{i,k}\right)=\lcm\left(e_{i,1},\dots,e_{i,k}\right)=e_i$.
\end{lemma}

\begin{proof}
For $k=1$ the statement is trivial, and for $k=2$ it is well known. 

Take $k>1$ and assume that the statement holds for sequences of less than $k$ integers. Then, $e'_{i,k}=\lcm\left(e_{i,1},\dots,e_{i,k-1}\right)$. Hence,
$$
\lcm\left(e_{i,1},\dots,e_{i,k}\right)=\lcm(e'_{i,k},e_{i,k})=e'_{i,k}e_{i,k}/d_{i,k}=e_{i,1} \cdots e_{i,k}/\left(d_{i,1}\cdots d_{i,k}\right).
$$
The identity $\lcm\left(e_{i,1},\dots,e_{i,k}\right)=e_i$ follows from $\ e_i\Z=\left\{e\in\Z\mid e\ga_i\in\g_{\mu_{i-1}}\right\}$. 
\end{proof}

\begin{lemma}\label{basisGi}
The following  family  $\iota_{i+1,1},\dots,\iota_{i+1,k}$ is a basis of $\g_i$:
$$\as{0.7}
(\iota_{i+1,1}\,\cdots\,\iota_{i+1,k})=(\iota_{i,1}\,\cdots\,\iota_{i,k})\,Q,\qquad
Q=\begin{pmatrix}
    d_{i,1}/e_{i,1}     &  & \lower.5ex\hbox{\hspace{-0.8cm} \bigZero} \\
\quad q_{m,j}   & \ddots &  \\
                     &    & d_{i,k}/e_{i,k}
\end{pmatrix},
$$
where $q_{m,j}= \ell_{i,j} e'_{i,j} h_{i,m}/ e_{i,m}$, for  $m>j$.
\end{lemma}


\begin{proof}
Consider the chain of $\Z$-modules $\g_{i-1}\subset\g_i\subset \Lambda$,
where $\Lambda$ is the $\Z$-module generated by $\iota_{i,1}/e_{i,1},\dots,\iota_{i,k}/e_{i,k}$. By Lemma \ref{lcm}, 
$$
\left(\Lambda\colon \g_i\right)=\left(\Lambda\colon \g_{i-1}\right)/\left(\g_i\colon \g_{i-1}\right)=e_{i,1} \cdots e_{i,k}/e_i=d_{i,1} \cdots d_{i,k}.
$$

On the other hand, let $\g'$ be the $\Z$-module generated by $\iota_{i+1,1},\dots,\iota_{i+1,k}$. Let us show that $\g'\subset \g_i$. 
In fact, by using the B\'ezout identities (\ref{bezout}) we get 
\begin{equation}\label{iotas}
\begin{array}{rl}
\iota_{i+1,j}\!\!&=\;\dfrac{d_{i,j}}{e_{i,j}}\,\iota_{i,j}+\ell_{i,j}e'_{i,j}\left(\dfrac{h_{i,j+1}}{e_{i,j+1}}\,\iota_{i,j+1}+\cdots+\dfrac{h_{i,k}}{e_{i,k}}\,\iota_{i,k}\right)\\&=\;\ell'_{i,j}\iota_{i,j}+\ell_{i,j}e'_{i,j}\left(\dfrac{h_{i,j}}{e_{i,j}}\,\iota_{i,j}+\dfrac{h_{i,j+1}}{e_{i,j+1}}\,\iota_{i,j+1}+\cdots+\dfrac{h_{i,k}}{e_{i,k}}\,\iota_{i,k}\right)\\&=\;
\ell'_{i,j}\iota_{i,j}+\ell_{i,j}e'_{i,j}\left(\ga_i-\dfrac{h_{i,1}}{e_{i,1}}\,\iota_{i,1}-\cdots-\dfrac{h_{i,j-1}}{e_{i,j-1}}\,\iota_{i,j-1}\right),
\end{array}
\end{equation}
for all $1\le j\le k$.
Now, by Lemma  \ref{lcm},
\begin{equation}\label{EE}
e'_{i,j}/e_{i,t}=\lcm(e_{i,1},\dots,e_{i,j-1})/e_{i,t}\in\Z,\qquad 1\le t<j. 
\end{equation}
Hence, (\ref{iotas}) expresses each $\iota_{i+1,j}$ as a $\Z$-linear combination of $\iota_{i,1},\dots,\iota_{i,j},\ga_i\in\g_i$.

Consider the chain of $\Z$-modules $\g'\subset\g_i\subset \Lambda$.

The lower triangular matrix $P=\di(e_{i,1},\dots,e_{i,k})Q$ has integer coefficients and $\det(P)=d_{i,1} \cdots d_{i,k}$. We may rewrite the claimed basis of $\g_i$ as:
$$
(\iota_{i+1,1}\,\cdots\,\iota_{i+1,k})=\left((1/e_{i,1})\,\iota_{i,1}\,\cdots\,(1/e_{i,k})\,\iota_{i,k}\right)P. 
$$
Hence, $(\Lambda\colon \g')=\det(P)=d_{i,1} \cdots d_{i,k}=(\Lambda\colon \g_i)$. This proves $\g'=\g_i$.
\end{proof}

\subsection{Rational functions of a MacLane chain}\label{subsecRat}

Our aim in this section is to construct an element $y_i\in\Delta_i$ which is transcendental over $\ka_i$ and satisfies $\Delta_i=\ka_i[y_i]$.

As indicated in Theorem \ref{kstructure}, we may take $$y_i=H_{\mu_i}(\phi_i)^{e_i}\epsilon_i,$$
for an arbitrary unit $\epsilon_i\in\gg_{\mu_i}^*$ of degree $-e_i\ga_i=-h_i$.
We shall construct this unit as the image in $\gg_{\mu_i}$ of an element in $K(x)$ with $\mu_i$-value equal to $-h_i$.

To this end, we construct in a recursive way rational functions $\pi_{i+1,j}\in K(x)^*$, for $-1\le i\le r$, whose $\mu_i$-values attain the chosen basis $\iota_{i+1,1},\dots,\iota_{i+1,k}$ of $\g_i$.

This will determine group homomorphisms
$$
\pi_{i+1}\colon \g_i\lra K(x)^*,\qquad \alpha\longmapsto \pi_{i+1}^\alpha,\qquad \mu_{i}(\pi_{i+1}^\alpha)=\alpha,
$$
where we agree that $\mu_{-1}=v$ and we define 
$$\pi_{i+1}^\alpha=\pi_{i+1,1}^{n_1}  \cdots  \pi_{i+1,k}^{n_k},\qquad \mbox{if }\ \alpha= n_1\, \iota_{i+1,1}+ \cdots + n_k\, \iota_{i+1,k},\quad n_1,\dots,n_k\in\Z.$$

\begin{definition}\label{ratfs} 
Choose arbitrary $\pi_{0,j}\in K$ such that $v(\pi_{0,j})=\iota_{0,j}$ for all $1\le j\le k$.

For $0\le i\le r$ and $1\le j\le k$ we define 
$$
Y_i=\phi_i^{e_i}\,\pi_i^{-h_i},\qquad
\pi_{i+1,j}=\left( \phi_i \, \pi_{i,1}^{-h_{i,1}/e_{i,1}} \cdots\, \pi_{i,j-1}^{-h_{i,j-1}/e_{i,j-1}}\right)^{\ell_{i,j}e'_{i,j}} \,  \pi_{i,j}^{\ell'_{i,j}}.
$$
In the definition of $\pi_{i+1,j}$, the rational functions $\pi_{i,1},\dots,\pi_{i,j-1}$ appear with integer exponents, because $e'_{i,j}/e_{i,t}$ is an integer for $t<j$, as shown in (\ref{EE}).
\end{definition}

For $i\ge0$, it is easy to deduce from the definition that:
\begin{equation}\label{phis}
\pi_{i+1,j}=a\,\phi_0^{n_0}\cdots\phi_{i-1}^{n_{i-1}}\phi_i^{\ell_{i,j}e'_{i,j}},
\end{equation}
for some $a\in K^*$ and certain (eventually negative) integer exponents $n_0,\dots, n_{i-1}$. Since $\phi_{i+1}\nmid_{\mu_i}\phi_\ell$ for $\ell\le i$, Lemma \ref{stable} shows that   
\begin{equation}\label{stability}
\mu_i(Y_i)=\mu_{i+1}(Y_i)=\cdots=\mu(Y_i),\quad \mu_i(\pi_{i+1,j})=\mu_{i+1}(\pi_{i+1,j})=\cdots=\mu(\pi_{i+1,j}).
\end{equation}
 
Clearly, $\mu_{i}(Y_i)=0$ by the definition of  $Y_i$. Let us compute the other stable value.

\begin{lemma}\label{values}For all $\ -1\le i\le r$, \ $1\leq j \leq k$, \ we have
$\ \mu_{i}(\pi_{i+1,j})=\iota_{i+1,j}$.
\end{lemma}

\begin{proof} 
For $i=-1$, $\mu_{-1}(\pi_{0,j})=\iota_{0,j}$ for all $j$ by definition. 
Suppose $i\ge0$ and the identity  holds for $i-1$. Then $\mu_i(\pi_{i,j})=\mu_{i-1}(\pi_{i,j})=\iota_{i,j}$ by (\ref{stability}). Hence,
$$
\mu_i(\pi_{i+1,j})=  l_{i,j}e'_{i,j} \left( \ga_i -\dfrac{h_{i,1}}{e_{i,1}}\, \iota_{i,1}-\cdots- \dfrac{h_{i,j-1}}{e_{i,j-1}}\, \iota_{i,j-1} \right) + l'_{i,j}\, \iota_{i,j}   = \iota_{i+1,j},                
$$
for all $1\le j\le k$, as shown in (\ref{iotas}).
\end{proof}

Our next aim is to establish certain relationships between the rational functions of Definition \ref{ratfs}. To this end we introduce some more notation.

\begin{definition}
Take an index $0\le i	\le r$. Denote:
$$
L'_i=\ell'_{i,1}\cdots \,\ell'_{i,k},\qquad
L_{i,j}=\ell_{i,j}\,\ell'_{i,j+1}\cdots \,\ell'_{i,k},\quad 1\le j\le k.
$$ 

Consider the function $L_i\colon (\g_{i-1})_\Q\lra \Q$ defined as
$$
L_i\left(\left(\iota_{i,1}\cdots \iota_{i,k}\right)\mathbf{v}\right)= \left(L_{i,1}\cdots L_{i,k}\right)\mathbf{v},\quad \forall\, \mathbf{v}\in\Q^{k\times 1}. 
$$
\end{definition}

\begin{lemma}\label{LH} 
For all \ $0 \leq i \leq r$, we have $L'_i+L_i(\ga_i) = 1/e_i$.
\end{lemma}

\begin{proof} The following identity is an immediate consequence of (\ref{bezout}):
\begin{equation}\label{queue}
\dfrac{\ell'_{i,j}}{e'_{i,j}} + \dfrac{\ell_{i,j}h_{i,j}}{e_{i,j}} = \dfrac{d_{i,j}}{e_{i,j}' e_{i,j}}= \dfrac{d_{i,1} \ldots d_{i,j}}{e_{i,1} \ldots e_{i,j}} =: \dfrac{1}{e'_{i,j+1}},\quad 1\leq j \leq k.
\end{equation}
Now, we claim that:
\begin{equation}\label{eqL's}
L'_i+L_{i,1} \,\dfrac{h_{i,1}}{e_{i,1}} + \cdots + L_{i,j}\, \dfrac{h_{i,j}}{e_{i,j}} = \ell'_{i,j+1} \ldots \ell'_{i,k}\,\frac{1}{e_{i,j+1}'},\quad 1\leq j \leq k.
\end{equation}
For $j=k$, this identity proves the lemma, since $e_{i,k+1}'=e_i$ by Lemma \ref{lcm}.

Let us apply an inductive argument.
For $j=1$, (\ref{eqL's}) follows directly from (\ref{queue}), having in mind that $e'_{i,1}=1$.
Now, if (\ref{eqL's}) holds for $j-1$:
$$L'_i+L_{i,1} \,\dfrac{h_{i,1}}{e_{i,1}} + \cdots + L_{i,j-1} \,\dfrac{h_{i,j-1}}{e_{i,j-1}} = \ell'_{i,j} \ldots \ell'_{i,k}\,\dfrac{1}{e_{i,j}'} ,$$
we deduce the identity (\ref{eqL's}) for $j$, just by adding $L_{i,j}h_{i,j}/e'_{i,j}$ to both sides of the equality, and by applying (\ref{queue}) to the right-hand side. 
\end{proof}

\begin{lemma}\label{LQ} For $0\leq i \leq r$, let $Q$ be the matrix introduced in Lemma \ref{basisGi}. Then,
$$e_i(L_{i,1}\, \dots\, L_{i,k}) \,Q = \ (\ell_{i,1}e_{i,1}' \,\dots\, \ell_{i,k} e_{i,k}').$$
\end{lemma}

\begin{proof}
The statement is equivalent to the following identity:
$$
L_{i,j}\, \dfrac{d_{i,j}}{e_{i,j}} + L_{i,j+1}\, q_{j+1,j} + \cdots + L_{i,k }\, q_{k,j} = \dfrac{1}{e_i} \, \ell_{i,j} e_{i,j}',\quad  1 \leq j \leq k,
$$
where $q_{m,j}=\ell_{i,j}e'_{i,j}h_{i,m}/e_{i,m}$ are the entries of $Q$ for $m>j$. Equivalently,
$$
L_{i,j} \ \dfrac{d_{i,j}}{\ell_{i,j}e_{i,j} e_{i,j}'} + L_{i,j+1} \ \dfrac{h_{i,j+1}}{e_{i,j+1}} + \cdots + L_{i,k } \ \dfrac{h_{i,k}}{e_{i,k}} = \dfrac{1}{e_i}.
$$

By Lemma \ref{LH}, this is equivalent to 
$$
L'_i+L_{i,1} \, \dfrac{h_{i,1}}{e_{i,1}} + \cdots + L_{i,j} \,  \dfrac{h_{i,j}}{e_{i,j}}= L_{i,j} \, \dfrac{d_{i,j}}{l_{i,j}e_{i,j} e_{i,j}'} = \ell_{i,j+1}' \cdots  \ell_{i,k}' \, \dfrac{1}{e_{i,j+1}'},$$
which was proven in (\ref{eqL's}).
\end{proof}

Let us rewrite the identities in Definition \ref{ratfs} by using formal logarithms:
\begin{equation}\label{log}
\left(\log \pi_{i+1,1} \cdots\, \log \pi_{i+1,k}\right)=\log\phi_i\,\left(\ell_{i,1}e'_{i,1}\,\cdots\,\ell_{i,k}e'_{i,k}\right)
+\left(\log \pi_{i,1} \cdots \,\log \pi_{i,k}\right)A,
\end{equation}
where $A=(a_{m,j})\in\Q^{k\times k}$ is the matrix with entries:
$$ 
a_{m,j}= 
\begin{cases}
    0,& \text{if } m>j\\
    \ell_{i,j}' , & \text{if } m=j \\
    -\ell_{i,j} e_{i,j}' h_{i,m}/e_{i,m} , & \text{if } m<j
\end{cases}.
$$

\begin{lemma}\label{vector} Let $0\le i \le r$. Consider the  matrix $B=\mathbf{u}\,(L_{i,1} \dots L_{i,k})\in\Q^{k\times k}$, where $\mathbf{u}$ is the column-vector defined in (\ref{u}).
\begin{enumerate}
\item[(1)] If $\alpha=(\iota_{i,1}\,\cdots\,\iota_{i,k})\,\mathbf{v}$, for some  $\mathbf{v}\in\Q^{k\times 1}$, then $B\mathbf{v}=L_i(\mathbf{\alpha})\mathbf{u}$. 
\item[(2)] $A=(I-e_iB)Q$.
\item[(3)] The vector $\mathbf{u}$ is an eigenvector of the matrix $I-e_iB$, with eigenvalue $e_iL'_i$.
\end{enumerate}
\end{lemma}

\begin{proof} 
Item (1) follows from the definition of the operator $L_i$.
By Lemma \ref{LQ}, 
$$(I-e_iB)Q=Q-e_i\mathbf{u}\,\left(L_{i,1} \dots L_{i,k}\right)Q=Q-\mathbf{u}\left(\ell_{i,1}e'_{i,1} \dots \ell_{i,k}e'_{i,k}\right).$$
Hence, item (2) is equivalent to 
$$q_{m,j} - a_{m,j}= \ell_{i,j} e_{i,j}'h_{i,m}/e_{i,m}, \quad \forall, m,j.$$
If $m>j$ ($a_{m,j} = 0$), or $m<j$ ($q_{m,j} = 0$), the identity is obvious. 
If $m=j$, it follows from $\ell_{i,m} h_{i,m} e_{i,m}'= d_{i,m} - \ell_{i,m}' e_{i,m}$.
Finally, 
$$(I-e_iB)\mathbf{u} = \mathbf{u} - e_i L_i(\ga_i)\mathbf{u}= e_iL'_i\,\mathbf{u},$$
by the first item and Lemma \ref{LH}. 
\end{proof}

\begin{proposition}\label{gammaL0} For all \ $0\leq i \leq r$, we have \ $\phi_i / \pi_{i+1}^{\ga_i} = Y_i^{L'_i}$.
\end{proposition}

\begin{proof}
By (\ref{log}) and Lemmas \ref{LH}, \ref{LQ} and \ref{vector},
\begin{align*}
\log\left( \phi_i/\pi_{i+1}^{\ga_i}\right)  =&\; \log\phi_i- \left(\log \pi_{i+1,1} \cdots \,\log \pi_{i+1,k}\right) \,Q^{-1} \mathbf{u} \\ 
=&\;  \left(1-\left(\ell_{i,1}e'_{i,1}\,\cdots\,\ell_{i,k}e'_{i,k}\right)\,Q^{-1}\mathbf{u}\right)\log \phi_i - \left(\log \pi_{i,1} \cdots \,\log \pi_{i,k}\right) \,AQ^{-1} \mathbf{u}\\
=&\;  (1-e_iL_i(\ga_i))\log \phi_i -\left(\log \pi_{i,1} \cdots\, \log \pi_{i,k}\right)(I-e_iB)\mathbf{u} \\
=&\;\log \phi_i^{e_iL'_i} -\left(\log \pi_{i,1} \cdots \,\log \pi_{i,k}\right)\,e_iL'_i\mathbf{u}=\log Y_i^{L'_i},
\end{align*}
because $\left(\log \pi_{i,1} \cdots \log \pi_{i,k}\right)\,e_i\mathbf{u}=\log \pi_i^{e_i\ga_i}=\log \pi_i^{h_i}$.
\end{proof}

The restriction of $\pi_{i+1}$ to the subgroup $\g_{i-1}\subset\g_i$ does not coincide with $\pi_i$. The next result computes the quotient of these two homomorphisms on $\g_{i-1}$.

\begin{proposition}\label{quotbeta} For all $0\le i\le r$ and $\beta \in \g_{i-1}$, we have
$\pi^{\beta}_{i+1}/\pi^{\beta}_i= Y_i^{L_i(\beta)}$.
\end{proposition}

\begin{proof} 
Let $\mathbf{b}$ be the column vector determined by $\beta=(\iota_{i,1} \cdots \iota_{i,k})\mathbf{b}$. 

By (\ref{log}) and Lemmas \ref{LQ} and \ref{vector},
\begin{align*}
\log \left(\pi_{i+1}^{\beta}/\pi^{\beta}_i\right) =&\;(\log \pi_{i+1,1} \cdots \log \pi_{i+1,k})\,Q^{-1} \mathbf{b}-\left(\log \pi_{i,1} \cdots \log \pi_{i,k}\right)\,\mathbf{b}\\
=&\;\log\phi_i\,\left(\ell_{i,1}e'_{i,1}\,\cdots\,\ell_{i,k}e'_{i,k}\right)\,Q^{-1} \mathbf{b}
+\left(\log \pi_{i,1} \cdots \log \pi_{i,k}\right)\,(AQ^{-1}-I) \mathbf{b}
\\
=&\;e_iL_i(\beta)\log \phi_i-\left(\log \pi_{i,1} \cdots \log \pi_{i,k}\right)\,e_iB 
\mathbf{b}\\
=&\;\log \phi_i^{e_iL_i(\beta)}-\left(\log \pi_{i,1} \cdots \log \pi_{i,k}\right)\,e_iL_i(\beta)\mathbf{u}=\log Y_i^{L_i(\beta)},
\end{align*}
because $\left(\log \pi_{i,1} \cdots \log \pi_{i,k}\right)\,e_i\mathbf{u}=\log \pi_i^{e_i\ga_i}=\log \pi_i^{h_i}$.
\end{proof}

The next result follows immediately from Propositions \ref{gammaL0} and \ref{quotbeta}.

\begin{proposition}\label{recurrence}
Let $(s,u)\in \Z_{\ge0}\times \g_{r-1}$. Then, $\phi_{r}^{s}\pi_{r}^{u}/\pi_{r+1}^{u+s\ga_r}=\left(Y_{r}\right)^{L'_{r}\,s - L_{r}\left(u\right)}$.
\end{proposition}

\subsection*{Images in $\ggm$ of the rational functions}
By (\ref{phis}), the rational functions  $Y_i,\, \pi_{i,j}$ introduced in Definition \ref{ratfs} are the product of some $a\in K^*$ by powers of $\phi_0,\dots,\phi_i$ with integer exponents. The exponent of $\phi_i$ is   $e_i$ and $0$, respectively. 

For $0\le j< i$, the exponent of $\phi_j$ may be negative, but the element $H_{\mu_{i}}(\phi_j)$ is a unit in $\gg_{\mu_i}$, by Proposition \ref{extension}.
Therefore, it makes sense to consider the image in the graded algebra of these rational functions, and the image of $\pi_{i,j}$ will be a unit.

\begin{definition}\label{ratfsGr}
For $0\le i\le r$, $1\leq j \leq k$, we define  
$$
x_i=H_{\mu_i}(\phi_i)\in \gg_{\mu_i}, \qquad y_i=H_{\mu_i}(Y_i)\in\Delta_i,\qquad p_{i,j}=H_{\mu_{i}}(\pi_{i,j})\in \gg_{\mu_i}^*.
$$
 
Also, we define group homomorphisms:
$$
p_i\colon\g_{i-1}\lra \gg_{\mu_i}^*,\qquad \alpha\longmapsto p_i^\alpha=H_{\mu_i}(\pi_i^\alpha)=p_{i,1}^{n_1}\cdots p_{i,k}^{n_k}, 
$$
if $\alpha=n_1\iota_{i,1}+\cdots+n_k\iota_{i,k}\in \g_{i-1}$, with $n_1,\dots,n_k\in\Z$.
\end{definition}


\begin{lemma}\label{j}
Let $(s,u),\,(s',u')\in\Z_{\ge0}\times\g_{\mu_{r-1}}$ such that $u+s\ga_r=u'+s'\ga_r$. 

Then,  there exists $j\in\Z$ such that
$$
s'=s+je_r,\quad u'=u-jh_r,\quad x_r^{s'}\,p_r^{u'}=x_r^s\,p_r^u\,y_r^j.
$$
\end{lemma}

\begin{proof}
From $(s'-s)\ga_r=u-u'\in\g_{\mu_{r-1}}$, we deduce $s'-s=je_r$ for some $j\in\Z$. Then, $u'=u-je_r\ga_r=u-jh_r$.
The lemma follows then from  $y_r=x_r^{e_r}\,p_r^{-h_r}$.
\end{proof}
\begin{definition}\label{xyz}
Suppose $0\le i<r$ and let $\alpha\in\g_{i-1}$. Since $\phi_{i+1}\nmid_{\mu_i}\phi_i$, Proposition \ref{extension} shows that the elements  $x_i,\,y_i,\,p_i^\alpha\in\gg_{\mu_i}$ are mapped to units in $\gg_{\mu_j}$, for all $j>i$, under the canonical homomorphism.  We denote these images respectively by  $$x_i\in\gg_{\mu_j}^*,\qquad z_i\in\Delta_j^*=\ka_j^*,\qquad p_i^\alpha\in\gg_{\mu_j}^*.$$ 
\end{definition}

\begin{lemma}\label{degpsi}
If $0\le i<r$, then $ \ka_{i+1}= \ka_i[z_i]= k[z_0,\dots,z_i]$.
\end{lemma}

\begin{proof}
By Theorem \ref{kstructure},  $\Delta_i= \ka_i[y_i]$. Hence,  $ \ka_{i+1}=\im\left(\Delta_i\to\Delta_{i+1}\right)= \ka_i[z_i]$.  
\end{proof}

In optimal MacLane chains, the elements $x_i,p_i^\alpha,y_r,z_i\in\ggm$ are ``almost" independent of the chain. Their precise variation is analyzed in section \ref{secDepchain}.

\section{Newton polygons}\label{secNewton}
\pagestyle{headings}
In this section, we study the Newton polygon operator attached to some $\mu\in\V$ with respect to a key polynomial. 
We generalize the results of \cite[Sec.2]{ResidualIdeals} (where $\mu$ was assumed to have rank one), up to a different normalization of the Newton polygons.

\subsection{Newton polygon operator}\label{secNP}
Consider two points $P=(s,\alpha),\ Q=(t,\beta)$ in the $\Q$-vector space $\qgq$.
The segment joining $P$ and $Q$ is the subset
$$
S=\left\{(s,\alpha)+ \ep\left(t-s,\beta-\alpha\right)\mid\  \ep \in  \Q,\ 0\le \ep\le 1\right\}\subset\qgq.
$$
If $P=Q$, then $S=\{P\}$.
If $s\ne t$, this segment has a natural slope
$$
\slp(S)=(\beta-\alpha)/(t-s)\in\gq.
$$


A subset of $\qgq$ is \emph{convex} if it contains the segment joining any two points in the subset.
The \emph{convex hull} of a finite subset $\cc\subset \qgq$ is the smallest convex subset of $\qgq$ containing $\cc$. 

The border of this hull is a sequence of chained segments.  If the points in $\cc$ have different abscissas, the leftmost and rightmost points are joined by two different chains of segments along the border, called the \emph{upper} and \emph{lower} convex hull of $\cc$.\e

Let $\mu\colon \kx\to \gq\cup\{\infty\}$ be a valuation in the space $\V$.
The choice of a key polynomial $\phi$ for $\mu$ determines a \emph{Newton polygon operator}
$$
\nphm\colon\, \kx\lra \pset\left({\qgq}\right),
$$
where $\pset\left({\qgq}\right)$ is the set of subsets of the rational space $\qgq$. 

The Newton polygon of the zero polynomial is the empty set. If 
$$g=\sum\nolimits_{0\le s}a_s\phi^s,\qquad \deg(a_s)<\deg(\phi)$$ is the canonical $\phi$-expansion of a non-zero $g\in \kx$, then $N:=\nphm(g)$ is defined to be the lower convex hull of the finite cloud of points $$\cc=\left\{(s,\mu(a_s))\mid s\in\Z_{\ge0}\right\}\subset\qgq.$$

Thus, $N$ is either a single point or a chain of segments, called the \emph{sides} of the polygon. From left to right, these sides have increasing slopes. 

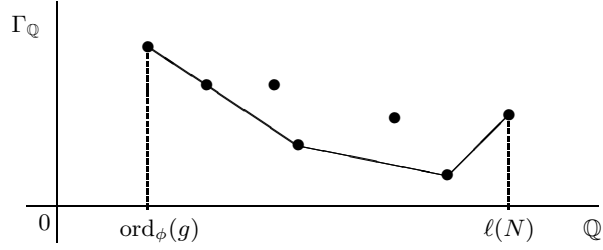
\begin{figure}
\caption{Newton polygon $N=\nphm(g)$ of $g\in \kx$}\label{figNmodel}
\begin{center}
\setlength{\unitlength}{4mm}
\begin{picture}(20,7.5)
\put(14.8,3.4){$\bullet$}\put(11,3.3){$\bullet$}\put(12.75,1.4){$\bullet$}\put(7,4.4){$\bullet$}
\put(7.8,2.4){$\bullet$}\put(4.75,4.4){$\bullet$}\put(2.8,5.65){$\bullet$}
\put(-1,0.6){\line(1,0){19}}\put(0,-0.6){\line(0,1){8}}
\put(8,2.6){\line(-3,2){5}}\put(8,2.63){\line(-3,2){5}}
\put(8,2.6){\line(5,-1){5}}\put(8,2.62){\line(5,-1){5}}
\put(13,1.6){\line(1,1){2}}\put(13,1.62){\line(1,1){2}}
\multiput(3,.5)(0,.25){22}{\vrule height2pt}
\multiput(15,.5)(0,.25){12}{\vrule height2pt}
\put(2.1,-0.4){\begin{footnotesize}$\ord_{\phi}(g)$\end{footnotesize}}
\put(14.2,-0.4){\begin{footnotesize}$\ell(N)$\end{footnotesize}}
\put(17.5,-0.4){\begin{footnotesize}$\Q$\end{footnotesize}}
\put(-1.5,6.4){\begin{footnotesize}$\gq$\end{footnotesize}}
\put(-.6,-.2){\begin{footnotesize}$0$\end{footnotesize}}
\end{picture}
\end{center}
\end{figure}

The abscissa of the left endpoint of $N$ is $s=\ord_{\phi}(g)$.  

The abscissa of the right endpoint of $N$ is called the \emph{length} of $N$, and is denoted:
$$
\ell(N)=\left\lfloor \deg(g)/\deg(\phi)\right\rfloor.
$$
The left and right endpoints of $N$, together with the points joining two different sides are called \emph{vertices} of $N$.

\begin{definition}\label{sla}
Let $g\in\kx$, $N=\nphm(g)$ and  $\la\in\gq$.

The \mbox{\emph{$\la$-component}} $S_\la(N)\subset N$ is the intersection of $N$ with the line of slope $-\la$ which first touches $N$ from below. In other  words,
$$S_\la(N)= \{(s,u)\in N\,\mid\, u+s\la\mbox{ is minimal}\,\}.$$

The abscissas of the endpoints of $S_\la(N)$ are denoted \ $s_{\mu,\phi,\la}(g)\le s'_{\mu,\phi,\la}(g)$. 
\end{definition}

If $N$ has a side $S$ of slope $-\la$, then $S_\la(N)=S$. Otherwise, $S_\la(N)$ is a vertex of $N$.  Figure \ref{figComponent0} illustrates both possibilities.

\begin{definition}\label{onesided}
We say that $N=\nphm(g)$ is \emph{one-sided} of slope $-\la$ if 
$$
N=S_\la(N),\qquad s_{\mu,\phi,\la}(g)=0,\qquad s'_{\mu,\phi,\la}(g)>0. 
$$
\end{definition}

\begin{figure}
\caption{$\la$-component of $\nphm(g)$. The line $L$ has slope $-\la$.}\label{figComponent0}
\begin{center}
\setlength{\unitlength}{4mm}
\begin{picture}(30,8.8)
\put(2.8,4.8){$\bullet$}\put(7.8,2.3){$\bullet$}
\put(-1,0.6){\line(1,0){13}}\put(0,-0.4){\line(0,1){9}}
\put(-1,7){\line(2,-1){11}}
\put(3,5){\line(-1,2){1.5}}\put(3,5.04){\line(-1,2){1.5}}
\put(3,5){\line(2,-1){5}}\put(3,5.04){\line(2,-1){5}}
\put(8,2.5){\line(4,-1){4}}\put(8,2.54){\line(4,-1){4}}
\multiput(3,.4)(0,.25){18}{\vrule height2pt}
\multiput(8,.4)(0,.25){8}{\vrule height2pt}
\put(7.3,-.3){\begin{footnotesize}$s'_{\mu,\phi,\la}(g)$\end{footnotesize}}
\put(2,-.3){\begin{footnotesize}$s_{\mu,\phi,\la}(g)$\end{footnotesize}}
\put(-1,7.4){\begin{footnotesize}$L$\end{footnotesize}}
\put(-.6,-.2){\begin{footnotesize}$0$\end{footnotesize}}
\put(3,7.2){\begin{footnotesize}$N=\nphm(g)$\end{footnotesize}}
\put(5.5,4){\begin{footnotesize}$S_\la(N)$\end{footnotesize}}
\put(20.8,5.4){$\bullet$}\put(23.8,3.3){$\bullet$}
\put(17,0.6){\line(1,0){13}}\put(18,-0.4){\line(0,1){9}}
\put(17,7.05){\line(2,-1){11}}
\put(24,3.6){\line(-3,2){3}}\put(24,3.64){\line(-3,2){3}}
\put(21,5.8){\line(-1,2){1}}\put(21,5.84){\line(-1,2){1}}
\put(24,3.55){\line(4,-1){4}}\put(24,3.59){\line(4,-1){4}}
\multiput(24,.5)(0,.25){12}{\vrule height2pt}
\put(20.3,-.3){\begin{footnotesize}$s_{\mu,\phi,\la}(g)=s'_{\mu,\phi,\la}(g)$\end{footnotesize}}
\put(17,7.3){\begin{footnotesize}$L$\end{footnotesize}}
\put(17.4,-.2){\begin{footnotesize}$0$\end{footnotesize}}
\put(22,7.1){\begin{footnotesize}$N=\nphm(g)$\end{footnotesize}}
\put(24,4){\begin{footnotesize}$S_\la(N)$\end{footnotesize}}
\end{picture}
\end{center}
\end{figure}
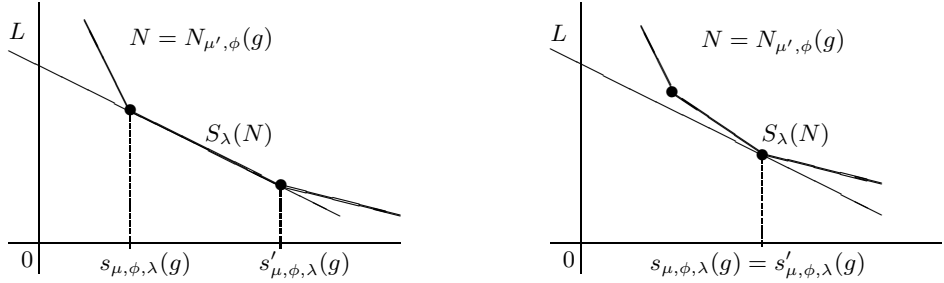

Since $\mu(g)=\mn\{\mu\left(a_s\phi^s\right)\mid s\ge0\}$, the next observation follows immediately.

\begin{remark}\label{mug}
For any non-zero $g\in\kx$, the value $\mu(g)\in\gq$ is the ordinate of the point where the vertical axis meets the line of slope $-\mu(\phi)$ containing the $\mu(\phi)$-component of the Newton polygon $\nphm(g)$. (see Figure \ref{figN-aug})
\end{remark}


Let $S$ be a side of $\nphm(g)$ with slope $\slp(S)<-\mu(\phi)$. Then, the augmented valuation $[\mu;\phi,-\slp(S)]$ contains relevant arihmetic information about $g$, with respect to $\phi$.

This motivates the next definition.

\begin{definition}
The \emph{principal Newton polygon} $\npphm(g)$ is the polygon formed by the sides of $\nphm(g)$ of slope less than $-\mu(\phi)$. 

If $\nphm(g)$ has no sides of slope less than $-\mu(\phi)$, then $\nphm^{\mbox{\tiny pp}}(g)$ is defined to be the left endpoint of $\nphm(g)$. 
\end{definition}

The length of a principal Newton polygon has an interesting algebraic interpretation in terms of the graded algebra $\ggm$, as shown in the next lemma.

\begin{lemma}\label{S0}
Let $N=\nphm(g)$ be the Newton polygon of a non-zero $g\in \kx$.
\begin{enumerate}
\item[(1)] The length $\ell\left(\nphm^{\mbox{\tiny pp}}(g)\right)$ 
coincides with the order with which the prime element $\hm(\phi)$ divides $\hm(g)$ in the graded algebra $\ggm$.
\item[(2)] If $h\in \kx$ satisfies $g\smu h$, then $S_{\mu(\phi)}(g)=S_{\mu(\phi)}(h)$.
\end{enumerate} 
\end{lemma}

\begin{proof}
The points $\left(s,\mu(a_s)\right)\in S_{\mu(\phi)}(g)$ are those with $s$ belonging to the set 
$$
I(g)=\left\{s\in\Z_{\ge0}\mid \mu(a_s\phi^s)=\mu(g)\right\}.
$$
Clearly, $\ell\left(\nphm^{\mbox{\tiny pp}}(g)\right)=\min(I)$. By \cite[Lem. 2.8]{KeyPol}, this non-negative integer is the order with which $\hm(\phi)$ divides $\hm(g)$ in  $\ggm$. 

Finally, if $g\smu h$, then  \cite[Lem. 2.10]{KeyPol} shows that $I(g)=I(h)$. Hence, the two segments  $S_{\mu(\phi)}(g)$, $S_{\mu(\phi)}(h)$ have endpoints with the same abscissas $s<s'$. 
On the other hand, the ordinates $u$, $u'$ of their endpoints  are determined by the abscissas and the common value $$\mu(g)=\mu(h)=u+s\mu(\phi)=u'+s'\mu(\phi).$$  
Hence, these segments coincide. 
\end{proof}

\subsection{Newton polygons with respect to augmented valuations}

Let us fix a valuation $\mu\in\V$ and a key polynomial $\phi\in \kpm$. Consider the augmentation
$$\mu'=[\mu;\phi,\ga],\qquad \ga=\mu'(\phi)>\mu(\phi),\quad \ga\in\gq.$$

Take a non-zero $\,g\in\kx$, with $\phi$-expansion $g=\sum_{0\le s}a_s\phi^s$. Let us denote
\begin{equation}\label{numu}
S_\ga(g)=S_\ga\left(\nph(g)\right),\qquad s(g)=s_{\mu,\phi,\ga}(g),\qquad s'(g)=s'_{\mu,\phi,\ga}(g),
\end{equation}
and let $u(g)\in \gq$ be the ordinate of the left endpoint of $S_\ga(g)$.

By Proposition \ref{extension}, $\phi$ is a key polynomial for $\mu'$ of minimal degree.

The Newton polygon $\nphm(g)$ is related to $\nph(g)$ in an obvious way:
$$
\nphm(g)=\nph(g),\qquad N^{\mbox{\tiny pp}}_{\mu',\phi}(g)\subsetneq \npphm(g).
$$
In fact, since $\mu'(a_s)=\mu(a_s)$ for all $s$, these polygons have the same cloud of points. But their 
principal parts are  different because $\mu'(\phi)>\mu(\phi)$.

Since $S_\ga(g)$ coincides with the $\mu'(\phi)$-component of $\nphm(g)$, the results in the last section show that $\nph(g)$ yields information about $g,\mu',\phi$.

\begin{figure}
\caption{$\nph(g)$ contains information about $\mu'=[\mu;\phi,\ga]$}\label{figN-aug}
\begin{center}
\setlength{\unitlength}{4mm}
\begin{picture}(21,9.5)
\put(11.8,5.8){$\bullet$}\put(10.8,3){$\bullet$}\put(7.8,5){$\bullet$}\put(2.75,5){$\bullet$}\put(6.75,3){$\bullet$}\put(4.75,4){$\bullet$}
\put(-1,0.5){\line(1,0){18}}\put(0,-0.5){\line(0,1){9.5}}
\put(3,5.2){\line(-1,3){1}}\put(3,5.23){\line(-1,3){1}}
\put(11,3.25){\line(-1,0){4}}\put(11,3.28){\line(-1,0){4}}
\put(7,3.2){\line(-2,1){4}}\put(7,3.23){\line(-2,1){4}}
\put(11,3.2){\line(1,1){3}}\put(11,3.22){\line(1,1){3}}
\multiput(3,.4)(0,.25){20}{\vrule height2pt}
\multiput(7,.4)(0,.25){11}{\vrule height2pt}
\put(-2,1.05){\line(6,1){17}}
\put(10,1.7){\line(-2,1){12}}
\put(2.4,-0.5){\begin{footnotesize}$s(g)$\end{footnotesize}}
\put(14,4.2){\begin{footnotesize}line of slope $-\mu(\phi)$\end{footnotesize}}
\put(10.5,1.5){\begin{footnotesize}line of slope $-\ga$\end{footnotesize}}
\put(16,-.4){\begin{footnotesize}$\Q$\end{footnotesize}}
\put(-1.4,8.3){\begin{footnotesize}$\gq$\end{footnotesize}}
\multiput(-.1,5.2)(.25,0){13}{\hbox to 2pt{\hrulefill }}
\put(6.5,-0.5){\begin{footnotesize}$s'(g)$\end{footnotesize}}
\put(-2,5){\begin{footnotesize}$u(g)$\end{footnotesize}}
\put(-2,1.6){\begin{footnotesize}$\mu(g)$\end{footnotesize}}
\put(-2,6.2){\begin{footnotesize}$\mu'(g)$\end{footnotesize}}
\put(-.6,-.4){\begin{footnotesize}$0$\end{footnotesize}}
\put(-.2,6.7){\line(1,0){.4}}
\put(-.2,1.45){\line(1,0){.4}}
\end{picture}
\end{center}
\end{figure}
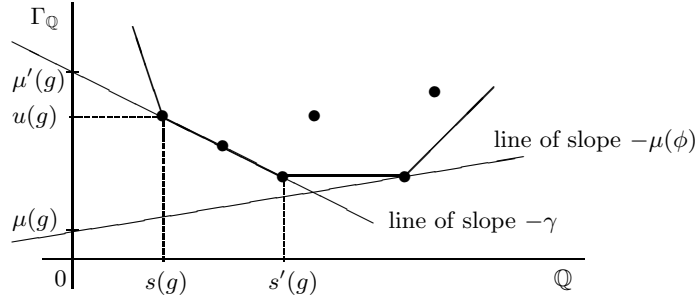



\begin{lemma}\label{additivity}
Let $g,h\in \kx$ be non-zero polynomials.
\begin{enumerate}
\item[(1)] The integer $s(g)$ is the order with which the prime element $H_{\mu'}(\phi)$ divides $H_{\mu'}(g)$ in the graded algebra $\gg_{\mu'}$. In particular,
$s(gh)=s(g)+s(h)$.
\item[(2)] $s'(gh)=s'(g)+s'(h)$. 
\end{enumerate}
\end{lemma}

\begin{proof}
Item (1) follows directly from Lemma \ref{S0}. Since $\phi$ is a key polynomial for $\mu'$ of minimal degree, item (2) follows from \cite[Cor. 3.4]{KeyPol}.
\end{proof}


\subsection{Addition of Newton polygons}\label{secAdd}
We admit that a point in the space $\qgq$ is a segment whose right and left endpoints coincide. 

There is a natural addition of segments in $\qgq$. The sum $S_1+S_2$ of two segments is the ordinary vector sum if at least one of the segments is a single point. Otherwise, $S_1+S_2$ is the  polygon whose left endpoint is the vector sum of the two left endpoints of $S_1, S_2$ and whose sides are the join of $S_1$ and $S_2$ considered with increasing slopes from left to right.

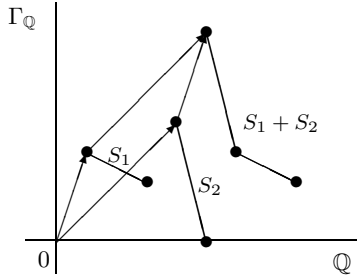
\begin{figure}
\caption{Addition of two segments}\label{figSum}
\begin{center}\setlength{\unitlength}{4mm}
\begin{picture}(10,9)
\put(-1,.6){\line(1,0){11}}\put(0,-.5){\line(0,1){9}}\put(0,.5){\vector(1,3){1}}
\put(.8,3.3){$\bullet$}\put(2.8,2.3){$\bullet$}
\put(1,3.5){\line(2,-1){2}}\put(1,3.52){\line(2,-1){2}}\put(0,.5){\vector(1,1){4}}
\put(3.75,4.3){$\bullet$}\put(4.75,0.3){$\bullet$}
\put(4,4.5){\line(1,-4){1}}\put(4,4.52){\line(1,-4){1}}\put(1,3.5){\vector(1,1){4}}\put(4,4.5){\vector(1,3){1}}
\put(4.75,7.3){$\bullet$}\put(5.75,3.3){$\bullet$}\put(7.75,2.3){$\bullet$}
\put(5,7.5){\line(1,-4){1}}\put(5,7.52){\line(1,-4){1}}\put(6,3.5){\line(2,-1){2}}\put(6,3.52){\line(2,-1){2}}\put(-.6,-.3){\begin{footnotesize}$0$\end{footnotesize}}\put(9.2,-.4){\begin{footnotesize}$\Q$\end{footnotesize}}\put(-1.6,7.8){\begin{footnotesize}$\gq$\end{footnotesize}}\put(1.7,3.2){\begin{scriptsize}$S_1$\end{scriptsize}}
\put(4.7,2.2){\begin{scriptsize}$S_2$\end{scriptsize}}\put(6.2,4.3){\begin{scriptsize}$S_1+S_2$\end{scriptsize}}
\end{picture}
\end{center}
\end{figure}

We keep dealing with an arbitrary valuation $\mu\in\V$, and a key polynomial $\phi\in\kpm$. 
Also, for any $\ga\in\gq$, $\ga>\mu(\phi)$, we keep using the notation of (\ref{numu}).

\begin{lemma}\label{sumSla}
For non-zero $g,h\in \kx$, and any $\ga\in\gq$, $\ga>\mu(\phi)$, we have $$S_\ga(gh)=S_\ga(g)+S_\ga(h).$$ 
\end{lemma}
    
\begin{proof}
Since the involved segments either have the same slope or consist of a single point, the statement is equivalent to the equalities
$$
s(gh)=s(g)+s(h),\quad s'(gh)=s'(g)+s'(h)\quad \mbox{and}\quad u(gh)=u(g)+u(h).
$$
The first two equalities follow from Lemma \ref{additivity}. 
In order to prove the third, consider the augmented valuation $\mu'=[\mu;\phi,\ga]$.
By Remark \ref{mug}, 
$$
u(g)+s(g)\ga=\mu'(g),\qquad
u(h)+s(h)\ga=\mu'(h),\qquad
u(gh)+s(gh)\ga=\mu'(gh).
$$
The claimed identity $u(gh)=u(g)+u(h)$ follows from $\mu'(gh)=\mu'(g)+\mu'(h)$ and $s(gh)=s(g)+s(h)$.
\end{proof}

The addition of segments may be extended to an addition law for Newton polygons, just by identifying a Newton polygon with the sum of its sides. 


As an immediate consequence of Lemma \ref{sumSla}, we get the \emph{Theorem of the product for principal Newton polygons}.

\begin{theorem}\label{product}
Let $\phi$ be a key polynomial for the valuation $\mu\in\V$. Then, for any non-zero $g,h\in \kx$ we have $\npphm(gh)=\npphm(g)+\npphm(h)$.\hfill{$\Box$}
\end{theorem}

The analogous statement for entire Newton polygons is false (cf. \cite[Sec. 2]{ResidualIdeals}).

\section[Residual polynomials operators]{Residual polynomial operators of inductive valuations}\label{secRi}

Consider a MacLane chain of an inductive valuation $\mu$ on $\kx$:
$$
\minf\ \stackrel{\phi_0,\ga_0}\lra\  \mu_0\ \stackrel{\phi_1,\ga_1}\lra\ \cdots
\ \stackrel{\phi_{r-1},\ga_{r-1}}\lra\ \mu_{r-1} 
\ \stackrel{\phi_{r},\ga_{r}}\lra\ \mu_{r}=\mu.
$$
We use the simplified notation \ $\gg=\ggm$, \ $\Delta=\Delta_\mu$, \
and we shall freely use all data associated with the MacLane chain in section \ref{secIndVals}. 

Specially relevant to our purpose is the tower of finite field extensions:
\begin{equation}\label{chaink}  k= \ka_0\subset  \ka_1\subset\cdots\subset  \ka_r;\qquad \ka_i=\im\left(\Delta_{i-1}\to\Delta_{i}\right)\subset\Delta_i,\quad 1\le i\le r.
\end{equation}
Each $\ka_i$ is the algebraic closure of $k$ in $\Delta_{i}$, and it satisfies $\Delta_{i}^*=\ka_i^*$.

Also, we shall make use of the Newton polygon operators
$$N_i:=N_{\mu_{i-1},\phi_i},\qquad 0\le i\le r.$$

\subsection{Definition of the residual polynomial  operator}\label{secDefR}

For a non-zero $ f \in K[x]$ consider the canonical $\phi_r$-expansion:
\begin{equation}\label{fexp}
f=\sum\nolimits_{s=0}^\ell a_s\phi_r^s,\qquad \deg(a_s)<\deg(\phi_r).
\end{equation}
By the definition of an augmented valuation, $$\mu(f)=\Min\{\mu\left(a_s\phi_r^s\right)\mid s\ge0\}=\Min\{\mu_{r-1}\left(a_s\right)+s\ga_r\mid s\ge0\}.$$
In the case $r=0$, we have $a_s\in K$ and we agree that $\mu_{-1}=v$.\e

The Newton polygon $N_r(f)$ is the lower convex hull of the cloud of points: 
$$\cc=\left\{Q_s\mid s\ge0\right\}\subset\Z_{\ge0}\times \g_{\mu_{r-1}},\qquad Q_s=\left(s,\mu_{r-1}\left(a_s\right)\right).
$$

Let $S_{\ga_r}(f)\subset \Q\times \gq$ be the $\ga_r$-component of $N_r(f)$ (Definition \ref{sla}). Let $$(s_0,u_0):=(s_r(f),u_r(f)),\qquad (s'_r(f),u'_r(f))$$ be the left and right endpoints of $S_{\ga_r}(f)$, respectively. 

By Remark \ref{mug}, for any point $(s,u)\in N_r(f)$, we have
$$
(s,u)\in S_{\ga_r}(f)\ \sii\ u+s\ga_r=\mu(f)=u_0+s_0\ga_r.
$$

By Lemma \ref{j}, $d:=(s'_r(f)-s_r(f))/e_r$ is an integer, and if we take
$$
s_j=s_0+je_r, \qquad u_j=u_0-je_r,\qquad P_j=(s_j,u_j),\qquad 0\le j\le d,
$$
then, $S_{\ga_r}(f)\cap \left(\Z_{\ge0}\times \g_{\mu_{r-1}}\right)=\left\{P_0,\,P_1,\, \dots,\,P_d\right\}$. 

By construction, the endpoints of $S_{\ga_r}(f)$ belong to the cloud $\cc$. That is, $P_0=Q_{s_0}$, $P_d=Q_{s_d}$. However, for $0<j<d$, the point $Q_{s_j}\in\cc$ may lie strictly above $P_j$. \e

In \cite[Sec. 5]{KeyPol}, a normalized (monic) residual polynomial of $f$ was considered:
$$
\widehat{R}_r(f)=\zeta_0+\zeta_1\,y+\cdots+\zeta_{d-1}y^{d-1}+y^d\in \ka_r[y],
$$
with coefficients 
$$\zeta_j=p_r^{-h_r(d-j)}\hm(a_{s_d})^{-1}\hm(a_{s_j})\in \ka_r^*,$$ 
if $Q_{s_j}$ lies on $N_r(f)$, and $\zeta_j=0$ otherwise.

The main aim of the paper is to find an efficient algorithm to compute the chain of fields (\ref{chaink}) and the polynomial $\widehat{R}_r(f)$. 
To this end, we consider a non-normalized version of the residual polynomial, defined only on a certain subset of $\kx$.

Consider the subset of \emph{polynomials having attainable $\mu$-values}:
$$
\kxfg=\left\{g\in\kx\mid \mu(g)\in\gmf\right\}\subset \kx,
$$ 
where $\gmf=\g_r\subset \gm$ is the finitely-generated subgroup introduced in section \ref{subsecChainFG}.

\begin{figure}
\caption{Newton polygon $N_r(f)$ for $f\in K[x]$. 
The line $L$ has slope $-\ga_r$}\label{figAlpha}
\begin{center}
\setlength{\unitlength}{5mm}
\begin{picture}(14,8)
\put(2.7,4.8){$\bullet$}\put(3.5,4.4){$\times$}\put(4.5,3.9){$\times$}\put(5.5,3.4){$\times$}\put(6.4,2.95){$\times$}\put(7.4,2.45){$\times$}\put(8.35,2.05){$\bullet$}\put(6.5,5.75){$\bullet$}
\put(-1,0.5){\line(1,0){15}}\put(0,-0.5){\line(0,1){8}}
\put(-1,7){\line(2,-1){12}}
\put(2.9,5){\line(-1,2){1}}\put(2.9,5.04){\line(-1,2){1}}
\put(3,5){\line(2,-1){4.5}}\put(3,5.04){\line(2,-1){5.5}}
\put(8.5,2.2){\line(4,-1){3}}\put(8.5,2.24){\line(4,-1){3}}
\multiput(2.9,.4)(0,.25){19}{\vrule height2pt}
\multiput(8.55,.4)(0,.25){8}{\vrule height2pt}
\multiput(6.7,.4)(0,.25){22}{\vrule height2pt}
\put(8.45,2.6){\begin{footnotesize}$P_d=Q_{s_d}$\end{footnotesize}}
\put(2.8,5.4){\begin{footnotesize}$P_0=Q_{s_0}$\end{footnotesize}}
\put(7,3.5){\begin{footnotesize}$P_j$\end{footnotesize}}
\put(6.6,6.5){\begin{footnotesize}$Q_{s_j}$\end{footnotesize}}
\put(6.5,-.3){\begin{footnotesize}$s_j$\end{footnotesize}}
\put(8.2,-.3){\begin{footnotesize}$s_d=s'_r(f)$\end{footnotesize}}
\put(.8,-.3){\begin{footnotesize}$s_r(f)=s_0$\end{footnotesize}}
\put(11.1,0.8){\begin{footnotesize}$L$\end{footnotesize}}
\put(-.5,-.3){\begin{footnotesize}$0$\end{footnotesize}}
\put(-3.6,4.8){\begin{footnotesize}$u_0=u_r(f)$\end{footnotesize}}
\multiput(-.1,5)(.25,0){13}{\hbox to 2pt{\hrulefill }}
\put(-.2,6.5){\line(1,0){.4}}
\put(-1.8,6.2){\begin{footnotesize}$\mu(f)$\end{footnotesize}}
\end{picture}
\end{center}
\end{figure}
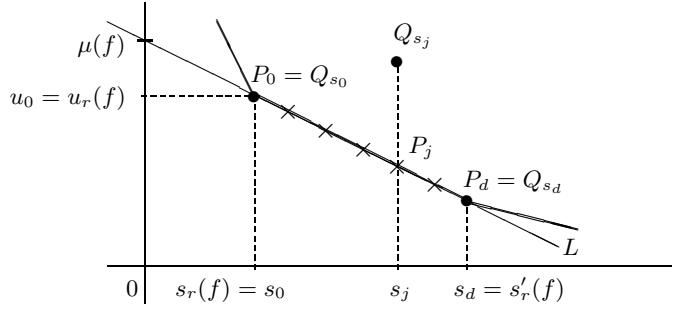

\begin{lemma}\label{agisatt}
Let $f\in\kx$ be a non-zero polynomial, with $\phi_r$-expansion as in (\ref{fexp}).
\begin{enumerate}
\item[(1)] There is a constant $a\in K^*$ such that $af$ has attainable $\mu$-value.
\item[(2)] If $f$ is monic and $\mu$-minimal, then it has attainable $\mu$-value.
\item[(3)] If $f$ has attainable $\mu$-value and $Q_s\in S_{\ga_r}(f)$, then the coefficient $a_s\in\kx$ has attainable $\mu_{r-1}$-value.

\end{enumerate}
\end{lemma}

\begin{proof}
The first item follows from $\g_\mu=\g+\gen{\ga_0,\dots,\ga_r}=\g+\g_r$.

If $f$ is monic and $\mu$-minimal, then $a_\ell=1$ and $\mu(f)=\mu(\phi_r^\ell)=\ell\ga_r\in\gmf$ by \cite[Prop. 3.7]{KeyPol}. 

The condition $Q_s\in S_{\ga_r}(f)$ implies $\mu(f)=\mu_{r-1}(a_s)+s\ga_r$. 
If $\mu(f)\in\g_r$, then  (\ref{intersection}) shows that $\mu_{r-1}(a_s)\in \g_{\mu_{r-1}}\cap \g_r=\g_{r-1}$.  
\end{proof}

\begin{definition}\label{defcj}
Let us define a residual polynomial operator:
$$
R_r\colon \kxfg\lra  \ka_r[y],\qquad f\longmapsto R_r(f)=c_0+c_1\,y+\cdots+c_d\,y^d.
$$
We agree that $R_r(0)=0$, and for a non-zero $f\in\kxfg$ as in (\ref{fexp}), we take
$$
c_j=\begin{cases}
p_r^{-\mu_{r-1}(a_{s_j})}\hm(a_{s_j})\in \ka_r^*,&\quad \mbox{ if }\ Q_{s_j}\mbox{ lies on }N_r(f),\\
0,&\quad \mbox{ otherwise}.
     \end{cases}
$$ 
\end{definition}

Since $\deg(a_{s_j})<\deg(\phi_r)$, $\hm(a_{s_j})$ is a unit in $\gg$ \cite[Prop. 3.5]{KeyPol}.
By Proposition \ref{extension} and Lemma \ref{agisatt}, $\mu(a_{s_j})=\mu_{r-1}(a_{s_j})\in\g_{r-1}$.
Therefore, if $Q_{s_j}$ lies on $N_r(f)$, the coefficient $c_j$ is a homogeneous unit of degree zero; that is, $c_j\in\Delta^*=\ka_r^*$.

If we normalize $R_r(f)$, we obtain the normalized residual polynomial $\Rh_r(f)$:
\begin{equation}\label{RRh}
c_d^{-1}R_r(f)=\Rh_r(f).
\end{equation}
In fact, if $Q_{s_j}$ lies on $N_r(f)$, then
$$
p_r^{-\mu_{r-1}(a_{s_j})}/p_r^{-\mu_{r-1}(a_{s_d})}=p_r^{\mu_{r-1}(a_{s_d})-\mu_{r-1}(a_{s_j})}=p_r^{u_d-u_j}=p_r^{-h_r(d-j)}.
$$

The operator $\Rh_r$ depends on the choice of $\phi$ and the unit $\epsilon=p_r^{-h_r}$, while $R_r$ depends on the choice of $\phi$ and the homomorphism $p_r$ from Definition \ref{ratfsGr}.

In section \ref{subsecRecursive}, we shall prove that the operator $R_r$ admits a recursive computation involving the previous operators $R_0,\dots,R_{r-1}$ attached to the intermediate valuations $\mu_0,\dots,\mu_{r-1}$ of the MacLane chain of $\mu$.\e

If $\g$ is finitely generated, we may take $\g_{-1}=\g$ as a universal choice for this subgroup. This implies $\g_i=\g_{\mu_i}$ for all $i$. Then, $\kxfg=\kx$ and the residual polynomial $R_r(f)$ is defined for all $f\in\kx$.

Otherwise, in any situation involving a finite number of polynomials, we may always assume that the subgroup $\g_{-1}$ is sufficiently large to allow the application of the operator $R_r$ to all the given polynomials.

\subsection{Basic properties of the operator $R_r$}\label{secBasicR}
The first basic properties of $R_r$ follow immediately from the fact that $c_0,c_d\ne0$.

\begin{lemma}\label{basicR}
Let $f\in\kxfg$ be a non-zero polynomial. Then,
$$\deg(R_r(f))=(s'_r(f)-s_r(f))/e_r\quad\mbox{ and }\quad R_r(f)(0)\in \ka_r^*.
$$
\end{lemma}

The essential property of the operator $R_r$ is described in the next result.

\begin{theorem}\label{Hmug}
For any $f\in\kxfg$, we have 
$$\hm(f)=x_r^{s_r(f)}p_r^{u_r(f)}R_r(f)(y_r).$$
\end{theorem}

\begin{proof}
Recall the notation $s_0=s_r(f)$, $u_0=u_r(f)$.
Consider the set of indices
$$
J=\left\{0\le j\le d\mid Q_{s_j}\mbox{ lies on }N_r(f)\right\}=\left\{0\le j\le d\mid \mu\left(a_{s_j}\phi_r^{s_j}\right)=\mu(f)\right\}.
$$
All indices $s\not\in \{s_j\mid j\in J\}$ have $\mu\left(a_s\phi_r^s\right)>\mu(f)$. Hence,
$$
f\sim_{\mu}f_J,\quad\mbox{ where }\quad f_J=\sum\nolimits_{j\in J}a_{s_j}\phi_r^{s_j}.
$$

If $f\in\kxfg$, Lemma \ref{agisatt} shows that   $u_0=\mu_{r-1}(a_{s_0})\in\g_{r-1}$. Thus, 
\begin{equation}\label{preMain}
f_J=\phi_r^{s_0}\pi_r^{u_0}\sum\nolimits_{j\in J}b_jY_r^j,\qquad b_j=\pi_r^{jh_r-u_0}a_{s_j}\in K(x),
\end{equation}
where $\pi_r^{u_0}$, $\;Y_r=\phi_r^{e_r}\pi_r^{-h_r}\in K(x)$ are the rational functions introduced in section \ref{subsecRat}. 

For all $j\in J$, we have $\mu_{r-1}(a_{s_j})=u_0-jh_r$. Therefore, 
$$
\hm(b_j)=p_r^{jh_r-u_0}\hm(a_{s_j})=p_r^{-\mu_{r-1}(a_{s_j})}\hm(a_{s_j})=c_j.
$$

Since $f\sim_{\mu}f_J$, the result follows from the application of $\hm$ to both sides of equation (\ref{preMain}), having in mind equation (\ref{Hmu}). 
\end{proof}

Thus, the homogeneous element $\hm(f)$ splits into a product of a power of the prime $x_r$, times a unit, times a degree-zero element $R_r(f)(y_r)\in\Delta$.

We now derive from Theorem \ref{Hmug} some more basic properties of the operator $R_r$.

\begin{theorem}\label{Delta}
Consider the $\ka_r$-algebra isomorphism $\ka_r[y]\simeq\Delta$ induced by $y\mapsto y_r$. For each $g\in\kx$ having $\mu(g)=0$, the inverse isomorphism assigms $$\hm(g)\;\longmapsto\; y^{s_r(g)/e_r}R_r(g).$$ 
\end{theorem}

\begin{proof}
Let $g\in K[x]$ with $\mu(g)=0$. 
By Lemma \ref{j}, applied to the pairs $(s_r(g),u_r(g))$ and $(0,0_\g)$, there exists an integer $j\ge0$ such that 
$$s_r(g)=je_r,\quad \mbox{and}\quad x_r^{s_r(g)}\,p_r^{u_r(g)}=y_r^j.
$$

By Theorem \ref{Hmug}, $\hm(g)= y_r^j\, R_{r}(g)(y_r)$. 
\end{proof}

\begin{corollary}\label{prodR} Let $f,g\in\kxfg$. Then, $R_r(fg)=R_r(f)R_r(g)$.
\end{corollary}

\begin{proof}
Since $\hm(fg)=\hm(f)\hm(g)$ and $p_r$ is a group homomorphism, the statement follows from  Theorems \ref{Hmug} and \ref{kstructure}, as long as:
$$
s_r(fg)=s_r(f)+s_r(g)\quad\mbox{ and }\quad u_r(fg)=u_r(f)+u_r(g).
$$
These identities follow from Lemma \ref{sumSla}.
\end{proof}

\begin{corollary}\label{equivR} Let $f,g\in\kxfg$. Then, 
$$
\as{1.3}
\begin{array}{ccl}
f \smu g&\iff&s_r(f)=s_r(g), \ u_r(f)=u_r(g) \ \mbox{ and }\ R_r(f)=R_r(g).\\
f \mmu g&\iff&s_r(f)\le s_r(g)\ \mbox{ and }\  R_r(f)\mid R_r(g) \ \mbox{ in }\  \ka_r[y].
\end{array}
$$
\end{corollary}

\begin{proof}
If $f \smu g$, then $S_{\ga_r}(f)=S_{\ga_r}(g)$ by Lemma \ref{additivity}. In particular, these segments have the same left endpoint: $(s_r(f),u_r(f))=(s_r(g),u_r(g))$. By Theorem \ref{Hmug}, $R_r(f)(y_r)=R_r(g)(y_r)$,  and we deduce
$R_r(f)=R_r(g)$ from Theorem \ref{kstructure}.

Conversely, the equalities $(s_r(f),u_r(f))=(s_r(g),u_r(g))$ and $R_r(f)=R_r(g)$ lead to $\hm(f)=\hm(g)$ by Theo\-rem \ref{Hmug}. 

The second equivalence follows from the analogous statement for the normalized operator $\widehat{R}_r$ \cite[Cor. 5.5]{KeyPol}.
\end{proof}

\begin{corollary}\label{Radd}
Let $f,g\in\kxfg$ such that $\mu(f)=\mu(g)$. Then,
$$
y^{\lfloor s_r(f)/e_r\rfloor}R_r(f)+y^{\lfloor s_r(g)/e_r\rfloor}R_r(g)=
\begin{cases}
y^{\lfloor s_r(f+g)/e_r\rfloor}R_r(f+g),&\mbox{ if }\mu(f+g)=\mu(f),\\
0,&\mbox{ if }\mu(f+g)>\mu(f). 
\end{cases}
$$
\end{corollary}

\begin{proof}
If $\mu(f+g)>\mu(f)$, we have $f\smu-g$. By Corollary \ref{equivR}, 
$$R_r(f)=R_r(-g)=-R_r(g) \quad\mbox{and}\quad  s_r(f)=s_r(-g)=s_r(g).$$ Hence,
$y^{\lfloor s_r(f)/e_r\rfloor}R_r(f)+y^{\lfloor s_r(g)/e_r\rfloor}R_r(g)=0$.\e

Denote $h=f+g$, and suppose that $\mu(h)=\mu(f)=\mu(g)$. Then,
$$\mu(f)=u_r(f)+s_r(f)\ga_r=u_r(g)+s_r(g)\ga_r=u_r(h)+s_r(h)\ga_r.$$
By Lemma \ref{j}, the following integer divisions  have a common remainder $0\le\ell<e_r$:
$$
s_r(f)=j_f\,e_r+\ell,\qquad
s_r(g)=j_g\,e_r+\ell,\qquad
s_r(h)=j_h\,e_r+\ell.
$$

Take $u=\mu(f)-\ell\ga_r=u_r(f)+j_fh_r\in\g_{r-1}$. By Lemma \ref{j},
\begin{equation}\label{xp}
x_r^{s_r(f)}\,p_r^{u_r(f)}=x_r^\ell\,p_r^u\,y_r^{j_f},\quad x_r^{s_r(g)}\,p_r^{u_r(g)}=x_r^\ell\,p_r^u\,y_r^{j_g},\quad x_r^{s_r(h)}\,p_r^{u_r(h)}=x_r^\ell\,p_r^u\,y_r^{j_h}.
\end{equation}

On the other hand, the identity from (\ref{Hmu}) and Theorem \ref{Hmug} show that
$$
x_r^{s_r(f)}p_r^{u_r(f)}R_r(f)(y_r)+x_r^{s_r(g)}p_r^{u_r(g)}R_r(g)(y_r)=
x_r^{s_r(h)}p_r^{u_r(h)}R_r(h)(y_r).
$$
By (\ref{xp}) and Theorem \ref{kstructure}, this identity is equivalent to
$$
y_r^{j_f}R_r(f)+y_r^{j_g}R_r(g)=
y_r^{j_h}R_r(h).
$$
This ends the proof, because  $j_f=\lfloor s_r(f)/e_r\rfloor$, $j_g=\lfloor s_r(g)/e_r\rfloor$, $j_h=\lfloor s_r(h)/e_r\rfloor$.
\end{proof}

\subsection*{Residual polynomial of a constant}
Take $a\in K^*$ with  $\al:=v(a)\in\g_{-1}$. Clearly,  $$S_{\ga_r}(a)=N_r(a)=\{(0,\al)\},\qquad s_r(a)=s'_r(a)=0.$$
By definition, $R_r(a)=p_r^{-\al}\hm(a)\in \ka_r^*$ is a constant.
If $r=0$, then $R_0(a)=\overline{a/\pi_0^{\al}}\in k^*$.
If $r>0$, we apply Proposition \ref{quotbeta}:
$$\pi^{\al}_{r}=\pi^{\al}_{r-1}\, Y_{r-1}^{L_{r-1}(\al)}\quad\imp\quad
p_r^{\al}=p_{r-1}^\al z_{r-1}^{L_{r-1}(\al)},$$
because $z_{r-1}$ is the image of $y_{r-1}=H_{\mu_{r-1}}(Y_{r-1})$ under the homomorphism $\gg_{\mu_{r-1}}\to\gg$.

By Proposition \ref{extension}, $\hm(a)$ is the image of $H_{\mu_{r-1}}(a)$ under the same homomorphism.  This leads to the following recurrence:
$$R_r(a)=z_{r-1}^{-L_{r-1}(\al)} R_{r-1}(a)= \cdots =z_{r-1}^{-L_{r-1}(\al)}\cdots \,z_0^{-L_0(\al)}R_0(a). $$

In particular, if $v(a)=0$, we have $R_r(a)=R_0(a)=\overline{a}\in k^*$.

\begin{lemma}\label{monicR}\mbox{\null}

\begin{enumerate}
\item[(1)] If $f\in\kx$ is monic and $\mu$-minimal, then $R_r(f)$ is a monic polynomial.
\item[(2)] $R_r(\phi_r^s)=1$ for any integer $s\ge0$.
\end{enumerate} 
\end{lemma}

\begin{proof}
If $f$ is monic and $\mu$-minimal, then $f\in\kxfg$ by Lemma \ref{agisatt}. 

The leading monomial of the $\phi_r$-expansion of $f$ is $\phi_r^{s'_r(f)}$ by \cite[Prop. 3.7]{KeyPol}.
Thus, for $d=\deg(R_r(f))$, we have 
$a_{s_d}=1$ and  
$c_d=p_r^{-\mu_{r-1}(1)}\hm(1)=1$.

Finally, $R_r(\phi_r)=1$ by definition. By Corollary \ref{prodR}, $R_r(\phi_r^s)=R_r(\phi_r)^s=1$.
\end{proof}

\subsection*{Existence of polynomials with a prescribed residual polynomial}

\begin{lemma}\label{Runicity}
Consider two polynomials $\varphi,\,\psi\in \ka_r[y]$ such that $\varphi(0)\ne0$, $\psi(0)\ne0$. Suppose that for two pairs $(s,u),\,(s',u')\in\Z_{\ge0}\times\g_{r-1}$ we have
\begin{equation}\label{su}
x_r^s\,p_r^u\,\varphi(y_r)=x_r^{s'}\,p_r^{u'}\,\psi(y_r).
\end{equation}

Then, $s=s'$, $u=u'$ and $\varphi=\psi$.
\end{lemma}

\begin{proof}
Since $\varphi(y_r)$ and $\psi(y_r)$ have degree zero in $\gg$, the equality (\ref{su}) implies
$$s\ga_r+u=\deg(x_r^s\,p_r^u)=\deg(x_r^{s'}\,p_r^{u'})=s'\ga_r+u'.$$

Suppose $s\le s'$. By Lemma \ref{j}, there exists an integer $j\ge0$ satisfying 
$$
s'=s+je_r,\quad u'=u-jh_r,\quad x_r^{s'}\,p_r^{u'}=x_r^s\,p_r^u\,y_r^j.
$$
Hence,  (\ref{su}) implies $\varphi(y_r)=y_r^j\,\psi(y_r)$.
By Theorem \ref{kstructure}, $\varphi=y^j\,\psi$. Since neither $\varphi$ nor $\psi$ are divisible by $y$, we have $j=0$. This implies  $s=s'$, $u=u'$ and $\varphi=\psi$.
\end{proof}

\begin{proposition}\label{Rconstruct}
Let $(s,u)\in\Z_{\ge0}\times\g_{r-1}$, and $\psi\in \ka_r[y]$ a polynomial with $\psi(0)\ne0$. 
Then, there exists a polynomial $f\in\kxfg$ such that
$$
s_r(f)=s,\qquad u_r(f)=u,\qquad R_r(f)=\psi.
$$
\end{proposition}

\begin{proof}
Take $f\in\kx$ such that $\hm(f)$ is the homogeneous element $x_r^s\,p_r^u\,\psi(y_r)\in\gg$. Since $\mu(f)=u+s\ga_r\in\g_r$, this polynomial has attainable $\mu$-value.
By Theorem \ref{Hmug}, 
$$
x_r^s\,p_r^u\,\psi(y_r)=\hm(f)=x_r^{s_r(f)}\,p_r^{u_r(f)}\,R_r(f)(y_r).
$$
The result follows from Lemma \ref{Runicity}.
\end{proof}

\subsection{Characterization of key polynomials for $\mu$}\label{secHomogP}

\begin{theorem}\label{charKP}
A monic $\phi\in\kx$ is a key polynomial for $\mu$ if and  only if one of the two following conditions is satisfied:
\begin{enumerate}
\item[(1)] $\deg(\phi)=\deg(\phi_r)$ \,and\; $\phi\smu\phi_r$.
\item[(2)] $s_r(\phi)=0$, $\deg(\phi)=e_rm_r\deg(R_r(\phi))$\, and\; $R_r(\phi)$ is irreducible in $ \ka_r[y]$.
\end{enumerate}

In case (1), $\rrm(\phi)=y_r\Delta$.
In case (2), 
$R_r(\phi)$ is monic, 
$\rrm(\phi)=R_r(\phi)(y_r)\Delta$,  and
$N_r(\phi)$ is one-sided of slope $-\ga_r$.
\end{theorem}

\begin{proof}
In case (2), $\phi$ is $\mu$-minimal by \cite[Prop. 3.7]{KeyPol}. Hence, $R_r(\phi)$ is monic by Lemma \ref{monicR}. On the other hand, we prove below that $N_r(\phi)$ is one-sided of slope $-\ga_r$.
The rest of statements were proved in \cite[Prop. 6.3]{KeyPol} for the normalized residual polynomial $\Rh_r(f)$. By equation (\ref{RRh}), they hold for $R_r(f)$ too.\e

Suppose that $\phi$ satisfies (2).
Since $S_{\ga_r}(\phi)\subset N_r(\phi)$, the equality $N_r(\phi)=S_{\ga_r}(\phi)$ will follow from the fact that the endpoints of both polygons have the same abscissas. 

Both left endpoints have abscissa $0$. Since $\ell(N_r(\phi))=\deg(\phi)/m_r=s'_r(\phi)$, both right endpoints have the same abscissa too.

Since $s'_r(\phi)>0$, $N_r(\phi)$ is one-sided of slope $-\ga_r$, according to Definition \ref{onesided}.
\end{proof}


Consider $\phi_1,\dots,\phi_r$ as key polynomials of $\mu_0,\dots,\mu_{r-1}$, respectively. By the definition of a MacLane chain, all these key polynomials fall in the second case of Theorem \ref{charKP}. 

\begin{corollary}\label{Rminimal}
Consider an index $0\le i< r$. 
\begin{enumerate}
\item[(1)] The Newton polygon $N_i(\phi_{i+1})$ is one-sided of slope $-\ga_i$.
\item[(2)] The residual polynomial  $R_{i}(\phi_{i+1})$ is the minimal polynomial of $z_{i}$ over $ \ka_{i}$. 

In particular, \ $\deg(R_{i}(\phi_{i+1}))=[ \ka_{i+1}\colon \ka_{i}]=f_{i}$.
\end{enumerate}
\end{corollary}

\begin{proof}
By Theorem \ref{charKP}, $N_i(\phi_{i+1})$ is one-sided of slope $-\ga_i$, and $R_{i}(\phi_{i+1})$ is monic irreducible. 
By Theorem \ref{Hmug}, $H_{\mu_{i}}(\phi_{i+1})$ is associate to $R_{i}(\phi_{i+1})(y_{i})$ in the graded algebra $\gg_{\mu_i}$. 
By Proposition \ref{extension}, the homomorphism $\gg_{\mu_i}\to \gg_{\mu_{i+1}}$ vanishes on both elements. 
Thus,  $R_{i}(\phi_{i+1})(z_{i})=0$, because it is the image of  $R_{i}(\phi_{i+1})(y_{i})$ under this homomorphism.
\end{proof}

Finally, we may deduce from these results a well-known property of key polynomials for inductive valuations: they are \emph{defectless} (see \cite{Vaq2}). 

\begin{corollary}\label{defect1phi}
A key polynomial $\phi\in\kpm$  satisfies $\deg(\phi)=e(\phi)f(\phi)$.

In particular, the valuation $v_\phi$ is the unique extension of $v$ to the field $K_\phi$.   
\end{corollary}

\begin{proof}
Consider a MacLane chain of $\mu$ as in (\ref{depth}). By Theorem \ref{charKP} we have two possibilities for a key polynomial $\phi$ for $\mu$. Let us discuss separatedly each case.

If $\deg(\phi)=\deg(\phi_r)$ and $\phi\smu\phi_r$, then Propositions \ref{vphi} and  \ref{maxsubfield} show that 
$$\g_{v_\phi}=\g_{\mu,\deg(\phi)}=\g_{\mu,\deg(\phi_r)}=\g_{v_{\phi_r}},\qquad  k_\phi\simeq k_{\phi_r}.$$ Thus, $e(\phi)=e(\phi_r)=e_0\cdots e_{r-1}$ and $f(\phi)=f(\phi_r)=f_0\cdots f_{r-1}$, by (\ref{ephi}) and (\ref{fphi}). 

On the other hand, Theorem \ref{charKP} and Corollary \ref{Rminimal}  show that
$$
\deg(\phi)=\deg(\phi_r)=e_{r-1}f_{r-1}m_{r-1}=e_{r-1}f_{r-1}\cdots\, e_0f_0=e(\phi_r)f(\phi_r)=e(\phi)f(\phi).
$$

Finally, for any $\ga\in\gq$, $\ga>\ga_r$, consider the augmented valuation $\mu'=[\mu;\phi,\ga]$. If $\phi\not\smu\phi_r$, we may extend our MacLane chain of $\mu$ to a MacLane chain of $\mu'$:
$$
\minf\ \stackrel{\phi_0,\ga_0}\lra\  \mu_0\ \stackrel{\phi_1,\ga_1}\lra\ \cdots
\ \stackrel{\phi_{r-1},\ga_{r-1}}\lra\ \mu_{r-1} 
\ \stackrel{\phi_{r},\ga_{r}}\lra\ \mu_{r}\ \stackrel{\phi,\ga}\lra\  \mu'.
$$
Hence,  $\deg(\phi)=e(\phi)f(\phi)$ by the same argument we used for $\phi_r$.
\end{proof}

\subsection{Recursive computation of the operator $R_r$}\label{subsecRecursive}

From an algorithmic perspective, Lemma \ref{degpsi} and Corollary \ref{Rminimal} show how to build the tower (\ref{chaink}) of residue fields 
with the irreducible polynomials $R_i(\phi_{i+1})\in \ka_i[x]$.
The fields may be constructed as $$\ka_{i+1}=\ka_i[x]/(R_i(\phi_{i+1})),$$ and we may identify the generator $z_i\in \ka_{i+1}$ with the class of $x$ in this quotient.

On the other hand, the algorithmic computation of the operator $R_r$ is based on explicit formulas for the coefficients of the residual polynomials $R_r(f)\in \ka_r[x]$, in terms of the previous operators $R_0,\dots,R_{r-1}$.

\begin{definition}\label{defeps}
For some $0\le i<r$, let $a\in\kx_{\mu_i\op{-att}}$. We define
$$
\ep_i(a)=(z_i)^{L'_i\,s_i(a)-L_i\left(u_i(a)\right)}\in  \ka_{i+1}^*,
$$
where $(s_i(a),u_i(a))$ is the left endpoint of $S_{\ga_i}(a)$, the $\ga_i$-component of $N_i(a)$.
\end{definition}

\begin{theorem}\label{recursivecj}
For $f=\sum_{0\le s}a_s\phi_r^s\in\kxfg$, let 
$R_r(f)=c_0+\cdots+c_dy^d\in \ka_r[y]$. For each $j$ such that $c_j\ne0$, denote $s_j=s_r(f)+je_r$. Then,
$$
\as{1.4}
c_j=\begin{cases}\overline{a_{s_j}/\pi_0^{v(a_{s_j})}},&\quad \mbox{if } r=0,\\
\ep_{r-1}(a_{s_j})\,R_{r-1}(a_{s_j})(z_{r-1}),&\quad \mbox{if } r>0.
\end{cases}
$$ 
\end{theorem}

\begin{proof}
If $r=0$, the statement is based on the equality 
$$p_0^{-v(a)}\hmz(a)=\overline{a/\pi_0^{v(a)}},\qquad \forall\,a\in K^*,
$$
which is a consequence of the identification $k=\ka_0$ established in section \ref{subsecML}.

Suppose $r>0$. It suffices to prove the equality
\begin{equation}\label{aimcj} 
p_r^{-\mu_{r-1}(a)}\hm(a)=\ep_{r-1}(a)\,R_{r-1}(a)(z_{r-1}),\qquad \forall\,a\in \kx\mbox{ with } \deg(a)<m_r.
\end{equation}

Take a non-zero $a\in \kx$ with $\deg(a)<m_r$. By Theorem \ref{Hmug},
\begin{equation}\label{ha}
H_{\mu_{r-1}}(a)=x_{r-1}^sp_{r-1}^uR_{r-1}(a)(y_{r-1}),\quad s=s_{r-1}(a_{s_j}),\ u=u_{r-1}(a_{s_j}).
\end{equation}

Since $\deg(a)<m_r=\deg(\phi_r)$, we have $\mu_{r-1}(a)=\mu(a)$, so that $\hm(a)$ is the image of $H_{\mu_{r-1}}(a)$ under the canonical homomorphism  $\gg_{\mu_{r-1}}\to\gg$.

By applying this homomorphism to the identity (\ref{ha}), we get
$$
\hm(a)=x_{r-1}^sp_{r-1}^uR_{r-1}(a)(z_{r-1}).
$$
Hence, the claimed identity (\ref{aimcj}) is equivalent to:
$$
p_r^{-\mu_{r-1}(a)}x_{r-1}^sp_{r-1}^u=\ep_{r-1}(a)=(z_{r-1})^{L'_{r-1}\,s-L_{r-1}\left(u\right)}.
$$
Since $\mu_{r-1}(a)=u+s\ga_{r-1}$, this identity follows from Proposition \ref{recurrence}, by applying $H_{\mu_{r-1}}$ to a similar identity between the corresponding rational functions.  
\end{proof}

This yields an algorithm to compute the operator $R_r$. 

Also, it is easy to deduce from Theorem \ref{recursivecj} an algorithm to compute polynomials in $\kx$ with a prescribed residual polynomial, in the spirit of Proposition \ref{Rconstruct}.

The latter algorithm may be used to construct key polynomials $\phi$ such that $R_r(\phi)$ is a prescribed monic irreducible polynomial in $\ka_r[y]$. 

\subsection{Dependence of $R_r$ on the choice of $\g_{-1}$ and its basis}\label{secDepgf}
Let  $\g'_{-1}\subset\g$ be  another finitely-generated subgroup satisfying the conditions of Definition \ref{defgf}, and let $\iota'_{0,1},\dots \iota'_{0,k'}$ be a $\Z$-basis of $\g'_{-1}$.
With respect to these choices, let  
$$
y'_i,\,(p'_i)^\alpha\in\gg_{\mu_i},\quad 0\le i\le r;\qquad
z'_i\in\gg^*_{\mu_{i+1}},\quad 0\le i< r,
$$ 
be the corresponding elements described in Definitions \ref{ratfsGr} and \ref{xyz}. 

Also, let $R'_r$ be the corresponding residual polynomial operator.

There is a natural family of group homomorphisms:
$$
\tau_i\colon \g_{i-1}\cap \g'_{i-1}\lra  \ka^*_i,\qquad \alpha\longmapsto (p'_i)^\alpha/p_i^\alpha,\qquad 0\le i\le r.
$$
In fact, this quotient $(p'_i)^\alpha/p_i^\alpha$ of two units belongs to
$\gg_{\mu_i}^*\cap\Delta_{i}=\Delta_{i}^*= \ka^*_i$.

\begin{lemma}\label{changeGfg}
$$
y_i=\tau_i(h_i)y'_i,\quad 0\le i\le r; \qquad z_i=\tau_i(h_i)z'_i,\quad0\le i< r.
$$
\end{lemma}

\begin{proof}
By construction, $h_i\in\g_{i-1}\cap \g'_{i-1}$.
Then, from $y_i'=x_i^{e_i}(p'_i)^{-h_i}$, $y_i=x_i^{e_i}\,p_i^{-h_i}$, we deduce $y_i=\tau_i(h_i)y'_i$. 
The other identity follows from this one by applying the homomorphism $\gg_{\mu_i}\to \gg_{\mu_{i+1}}$.
\end{proof}

\begin{theorem}\label{RchangeGfg}
Suppose that $f\in\kx$ has $\mu(f)\in\gmf\cup \left(\g'_\mu\right)^{\operatorname{fg}}$. Then,
$$
R_r(f)(y)=\xi\,R'_r(f)(\zeta y),\qquad \mbox{where } \quad \xi=\tau_r\left(u_r(f)\right), \quad\zeta=\tau_r\left(-h_r\right)\in \ka_r^*.
$$
\end{theorem}

\begin{proof}
Let us denote for simplicity $s=s_r(f)$, $u=u_r(f)$. By Theorem \ref{Hmug},
$$
x_r^s\,p_r^u\,R_r(f)(y_r)=\hm(f)=x_r^s\,(p'_r)^u\,R'_r(f)(y'_r).
$$
By Lemma \ref{changeGfg}, this is equivalent to
$$
R_r(f)(y_r)=\tau_r(u)\,R'_r(f)\left(\tau_r(-h_r)y_r\right).
$$
By Theorem \ref{kstructure}, $R_r(f)(y)=\tau_r(u)\,R'_r(f)\left(\tau_r(-h_r)y\right)$.
\end{proof}


\subsection[Dependence of \mbox{$R_r$} on the MacLane chain]{Dependence of $R_r$ on the choice of an optimal MacLane chain}\label{secDepchain}
In this section, we discuss the variation of the elements $x_r,\,y_r,\,z_{r-1},\,p_{r+1}^\alpha\in\gg$ and the operator $R_r$, when we consider different optimal MacLane chains of $\mu$. 

By Proposition \ref{unicity}, two optimal MacLane chains of $\mu$ have the same length $r$, the same intermediate valuations $\mu_0,\dots,\mu_r$, and the same values $\ga_0,\dots,\ga_r\in\gq$. They may differ only in the choice of the key polynomials,
which must be related as follows:
$$
\phi^*_i=\phi_i+a_i,\qquad \deg(a_i)<m_i,\qquad \mu_i(a_i)\ge\ga_i,\quad 0\le i\le r.
$$
In particular, both chains support the same invariants 
$m_i,\,e_i\in\Z_{>0}$, $h_i\in\g_{i-1}$.

\begin{lemma}\label{ei=1}
If $\phi^*_r\not\smu\phi_r$, then $e_r=1$.
\end{lemma}

\begin{proof}
The condition  $\phi^*_r\not\smu\phi_r$ is equivalent to $\mu(a_r)=\mu(\phi_r)=\ga_r$. Since $\mu(a_r)=\mu_{r-1}(a_r)\in\g_{\mu_{r-1}}$, this leads to $\ga_r\in \g_{\mu_{r-1}}$, which implies $e_r=1$.
\end{proof}


We mark with a superscript $(\ )^*$ all data and operators attached to the MacLane chain determined by the choice of the key polynomials $\phi_0^*,\dots,\phi_r^*$.

\begin{theorem}\label{lastlevel}
With the above notation.
\begin{enumerate}
\item[(1)] The group homomorphisms $p^*_{i}$ and $p_{i}$ coincide for all \ $0\le i\le r+1$.
\item[(2)] If $\phi^*_r\smu\phi_r$, then \ $x^*_r=x_r,\quad y^*_r=y_r,\quad R^*_r=R_r$.
\item[(3)] If $\phi^*_r\not\smu\phi_r$, then \ $x^*_r=x_r+p_r^{h_r}\eta,\quad y^*_r=y_r+\eta$, where \,$\eta=R_r(a_r)\in \ka_r^*$.

Moreover, for a non-zero $g\in \kxfg$ we have
\begin{equation}\label{RR*}
y^{s_r(g)}R_r(g)(y)=(y+\eta)^{s^*_r(g)}R^*_r(g)(y+\eta).
\end{equation}
\end{enumerate}
\end{theorem}

\begin{proof}
Let us first prove by induction on $r$ all statements concerning $p^*_{i}$, $x^*_r$ and $y^*_r$.

By Definition \ref{ratfsGr}, $p^*_{0}=p_{0}$, because the choice of $\pi_{0,j}\in K^*$ is independent of the MacLane chain. Suppose $p^*_{i}=p_{i}$ for all $0\le i\le r$.

If $\phi^*_r\smu\phi_r$, then $x_r^*=\hm(\phi_r^*)=\hm(\phi_r)=x_r$, leading to
 $y^*_r=y_r$ and $p^*_{r+1,j}=p_{r+1,j}$ for all $j$, by Definitions \ref{ratfs} and \ref{ratfsGr}. Hence, $p^*_{r+1}=p_{r+1}$.

If $\phi^*_r\not\smu\phi_r$, then $e_r=1$ by Lemma \ref{ei=1}. By (\ref{bezout}), $\ell_{r,j}=0$, $\ell'_{r,j}=1$ for all $j$. Thus,
$$
\pi^*_{r+1,j}=\pi^*_{r,j},\qquad \pi_{r+1,j}=\pi_{r,j},\quad 1\le j\le k,
$$
by Definition \ref{ratfs}.
Hence, $p^*_{r+1}=p^*_{r}=p_{r}=p_{r+1}$.

Also, since $\deg(a_r)<m_r$, we get 
$$
s_r(a_r)=0,\qquad u_r(a_r)=\mu_{r-1}(a_r)=\mu(a_r)=\mu(\phi_r)=\ga_r=h_r.
$$
Hence, $a_r$ has an atainable $\mu$-value, and Theorem \ref{Hmug} shows that 
$$
\hm(a_r)=p_r^{h_r}R_r(a_r)=p_r^{h_r}\eta.
$$
By using equation (\ref{Hmu}), we get in this case 
$$
x_r^*=\hm(\phi_r^*)=\hm(\phi_r)+\hm(a_r)=x_r+p_r^{h_r}\eta,
$$
leading to $y^*_r=x^*_r(p^*_r)^{-h_r}=x^*_r(p_r)^{-h_r}=x_rp_r^{-h_r}+\eta=y_r+\eta$.\e

Finally, let us prove the statements concerning $R_r$.

Let $g\in\kx$ be a non-zero polynomial with attainable $\mu$-value. \e

\noindent{\bf Case }$\phi^*_r\smu\phi_r$. 
Since $x^*_r=x_r$, $y^*_r=y_r$, and $p^*_r=p_r$, Theorem \ref{Hmug} shows that
$$
x_r^{s_r(g)}\,p_r^{u_r(g)}R_r(g)(y_r)=\hm(g)=x_r^{s^*_r(g)}\,p_r^{u^*_r(g)}R^*_r(g)(y_r).
$$
By Lemma \ref{Runicity}, $s^*_r(g)=s_r(g)$, $u^*_r(g)=u_r(g)$, and $R^*_r(g)=R_r(g)$.\e



\noindent{\bf Case }$\phi^*_r\not\smu\phi_r$. Recall that $e_r=1$ and $p^*_r=p_r$. By Theorem \ref{Hmug},
$$
x_r^{s_r(g)}\,p_r^{u_r(g)}R_r(g)(y_r)=\hm(g)=(x^*_r)^{s^*_r(g)}\,p_r^{u^*_r(g)}R^*_r(g)(y^*_r).
$$
Since $x_r=p_r^{h_r}y_r$ and $x^*_r=p_r^{h_r}y^*_r=p_r^{h_r}(y_r+\eta)$, we deduce 
$$
y_r^{s_r(g)}p_r^{u_r(g)+s_r(g)h_r}R_r(g)(y_r)=(y_r+\eta)^{s^*_r(g)}p_r^{u^*_r(g)+s^*_r(g)h_r}R^*_r(g)(y_r+\eta).
$$
Since $u_r(g)+s_r(g)h_r=\mu(g)=u^*_r(g)+s^*_r(g)h_r$, we may drop the powers of $p_r$:
$$
y_r^{s_r(g)}R_r(g)(y_r)=(y_r+\eta)^{s^*_r(g)}R^*_r(g)(y_r+\eta).
$$
This proves (\ref{RR*}), as a consequence of Theorem \ref{kstructure}.
\end{proof}



\end{document}